\documentclass[11pt]{amsart}

\usepackage{times}
\usepackage{amsmath,amsfonts,amssymb,amsopn,amscd,amsthm}
\usepackage{graphicx,xspace}
\usepackage{epsfig}
\usepackage{enumerate} 
\usepackage{enumitem}
\usepackage{xcolor}
\usepackage[all]{xypic}

\usepackage[T1]{fontenc}
\usepackage[utf8]{inputenc}

\usepackage{hyperref}
\usepackage{psfrag}
\usepackage{bm}
\usepackage{youngtab}
\usepackage{tikz}

\theoremstyle{plain}
\newtheorem{lemma}{Lemma}[section]
\newtheorem{theorem}[lemma]{Theorem}

\theoremstyle{remark}
\newtheorem{remark}[lemma]{Remark}

\newtheorem{definition}[lemma]{Definition}

\author[O. Tout]{Omar Tout}
\address{LaBRI, Universit\'e Bordeaux 1, 351 cours de la Lib\'eration, 33 400
Talence, France}
\email{omar.tout@labri.fr}

\title[Structure coefficients of the Hecke algebra of $(\mathcal{S}_{2n},\mathcal{B}_n)$]{
Structure coefficients of the Hecke algebra of $(\mathcal{S}_{2n},\mathcal{B}_n)$}

\thanks{Partially supported by ANR grant PSYCO ANR-11-JS02-001}
\usetikzlibrary{snakes}
\usepackage[all]{xy}
\usepackage{varioref}
\theoremstyle{plain}
\newtheorem{theoreme}{Theorem}[section]
\newtheorem{prop}[theoreme]{Proposition}
\newtheorem{ex}{Example}[section]
\newtheorem{cor}[theoreme]{Corollary}
\newtheorem{lem}[theoreme]{Lemma}

\theoremstyle{definition}
\theoremstyle{remark}
\newtheorem*{rem}{Remark}
\newtheorem{obs}[subsection]{Observation}
\newtheorem*{notation}{Notation}
\newtheorem*{fact}{Fact}

\date{}
\def\pr{\ast}
\DeclareMathOperator{\supp}{supp}
\DeclareMathOperator{\Id}{Id}
\begin{document}
\maketitle
\paragraph{Abstract.} The Hecke algebra of the pair $(\mathcal{S}_{2n},\mathcal{B}_n)$, where $\mathcal{B}_n$ is the hyperoctahedral subgroup of $\mathcal{S}_{2n}$, was introduced by James in 1961. It is a natural analogue of the center of the symmetric group algebra. In this paper, we give a polynomiality property of its structure coefficients. Our main tool is a combinatorial algebra which projects onto the Hecke algebra of $(\mathcal{S}_{2n},\mathcal{B}_n)$ for every $n$. To build it, by using partial bijections we introduce and study a new class of finite dimensional algebras.
\section{Introduction}
For a positive integer $n$, let $\mathcal{S}_n$ denote the symmetric group of permutations on the set $[n]:=\lbrace 1,2,\cdots,n\rbrace$, and let $\mathbb{C}[\mathcal{S}_n]$ denote the group-algebra of $\mathcal{S}_n$ over $\mathbb{C}$, the field of complex numbers. The center of $\mathbb{C}[\mathcal{S}_n]$, denoted by $Z(\mathbb{C}[\mathcal{S}_n])$ is a classical object in combinatorics. It is linearly generated by elements $\mathcal{Z}_\lambda$, indexed by partitions of $n$, where $\mathcal{Z}_\lambda$ is the sum of permutations of $[n]$ with cycle-type $\lambda$. The structure coefficients $c_{\lambda\delta}^{\rho}$ describe the product in this algebra, they are defined by the equation:
$$\mathcal{Z}_\lambda \mathcal{Z}_\delta=\sum_{\rho\text{ partition of } n}c_{\lambda\delta}^{\rho}\mathcal{Z}_\rho.$$
In other words, $c_{\lambda\delta}^{\rho}$ counts the number of pairs of permutations $(x,y)$ with cycle-type $\lambda$ and $\delta$ such that $x\cdot y=z$ for a fixed permutation $z$ with cycle-type $\rho$.
It is known that these coefficients count the number of embeddings of certain graphs into orientable surfaces  (see \cite{CoriHypermaps}). One of the tools used to calculate these coefficients is the representation theory of the symmetric group, see \cite[Lemma 3.3]{JaVi90}. In \cite[Theorem 2.1]{GoupilSchaefferStructureCoef}, Goupil and Schaeffer have a formula for $c_{\lambda\delta}^{\rho}$ if one of the partitions $\lambda,\delta$ and $\rho$ is equal to $(n)$. There are no formulas for $c_{\lambda\delta}^{\rho}$ in general.

In 1958, Farahat and Higman proved the polynomiality of the coefficients $c_{\lambda\delta}^{\rho}$ in $n$ when $\lambda$, $\delta$ and $\rho$ are fixed partitions, completed with parts equal to $1$ to get partitions of $n$, \cite[Theorem 2.2]{FaharatHigman1959}. More recently, in \cite{IvanovKerov1999}, by using objects called partial permutations, the same result is obtained by Ivanov and Kerov. This more recent proof provides a combinatorial description of the coefficients of the relevant polynomials.
\bigskip

Here, we consider the Hecke algebra of the pair $(\mathcal{S}_{2n},\mathcal{B}_n)$, denoted by $\mathbb{C}[\mathcal{B}_n\backslash \mathcal{S}_{2n}/\mathcal{B}_n]$, where $\mathcal{B}_n$ is the hyperoctahedral group. It was introduced by James in \cite{James1961} and it also has a basis indexed by partitions of $n$. The algebra $\mathbb{C}[\mathcal{B}_n\backslash \mathcal{S}_{2n}/\mathcal{B}_n]$ is a natural analogue of $Z(\mathbb{C}[\mathcal{S}_n])$ for several reasons. Goulden and Jackson proved in \cite{GouldenJacksonLocallyOrientedMaps} that its structure coefficients count graphs drawn on non-oriented surfaces. To get formulas for these coefficients, zonal characters are used instead of irreducible characters of the symmetric group, see  \cite[Section VII, 2]{McDo}.

In this paper we give a polynomiality property of the structure coefficients of the Hecke algebra of $(\mathcal{S}_{2n},\mathcal{B}_n)$. Namely, we prove that these coefficients can be written as the product of the number $2^nn!$ with a polynomial in $n$. In some specific basis, this polynomial has non-negative coefficients that have a combinatorial interpretation. Moreover, we are able to give an upper bound for its degree . Our proof is based on the construction of an universal algebra which projects onto the Hecke algebra of $(\mathcal{S}_{2n},\mathcal{B}_n)$ for every $n$\footnote{In this sense, we shall call it a universal algebra.}. This method was already used by Ivanov and Kerov in \cite{IvanovKerov1999}. What is original in our approach is that the product in our universal algebra is computed as an average of combinatorial objects called partial bijections of $n$. Recently, P.-L. Méliot has used this same idea in \cite{meliot2013partial} to give a polynomiality property for the structure coefficients of the center of the group-algebra $\mathbb{C}[GL(n,\mathbb{F}_q)]$, where $GL(n,\mathbb{F}_q)$ is the group of invertible $n\times n$ matrices with coefficients in $\mathbb{F}_q$. Because of the similarities between Meliot's construction and ours, we are convinced we should build a general framework in which such a result (polynomiality of the structure coefficients) always holds. This is the subject of future work.

A weaker version of our polynomiality result (without non-negativity of the coefficients) for the structure coefficients of Hecke algebra of $(\mathcal{S}_{2n},\mathcal{B}_n)$ has been established by an indirect approach using Jack polynomials in \cite[Proposition 4.4]{dolega2012kerov}. There is no combinatorial description in that proof. By a different approach than ours, in \cite{Aker20122465}, Aker and Can study the Hecke ring $\mathbb{C}[\mathcal{B}_n\backslash \mathcal{S}_{2n}/\mathcal{B}_n]$, however, it seems that there is a minor issue in their proof of polynomiality of the structure coefficients (\cite{Can}). A universal algebra also appears in this paper, but it does not have a combinatorial realization as ours.

As explained, our proof goes through the construction of an universal algebra which projects onto the Hecke algebra of $(\mathcal{S}_{2n},\mathcal{B}_n)$ for every $n$. We are able to give a link between this algebra, and the algebra of \textit{shifted symmetric functions}. Shifted symmetric functions have been introduced and studied by A. Okounkov and G. Olshansky in 1996, see \cite{1996q.alg.....8020O}. They are deformations of usual symmetric functions that display remarkable properties. 
\bigskip

The paper is organized as follows. In Section \ref{section 2}, we review all necessary definitions to describe the Hecke algebra of $(\mathcal{S}_{2n},\mathcal{B}_n)$. Then, we state our main result about its structure coefficients. We start Section \ref{section 3} by introducing partial bijections of $n$ then we build our universal algebra. We use this algebra in Section \ref{section 4} to prove our main result, Theorem \ref{Theorem 2.1}. In Section \ref{section 5}, we show how the universal algebra is related with the algebra of shifted symmetric functions and in Section \ref{section6} we exhibit some filtrations on this universal algebra, which implies the above mentioned upper bounds for the degree of the polynomials.  
\section{Definitions and statement of the main result}\label{section 2}
\subsection{Partitions}
Since partitions index bases of the algebras studied in this paper, we recall the main definitions. A \textit{partition} $\lambda$ is a list of integers $(\lambda_1,\ldots,\lambda_l)$ where $\lambda_1\geq \lambda_2\geq\ldots \lambda_l\geq 1.$ The $\lambda_i$ are called the \textit{parts} of $\lambda$; the \textit{size} of $\lambda$, denoted by $|\lambda|$, is the sum of all of its parts. If $|\lambda|=n$, we say that $\lambda$ is a partition of $n$ and we will write $\lambda\vdash n$. The number of parts of $\lambda$ is denoted by $l(\lambda)$. We will also use the exponential notation $\lambda=(1^{m_1(\lambda)},2^{m_2(\lambda)},3^{m_3(\lambda)},\ldots)$, where $m_i(\lambda)$ is the number of parts equal to $i$ in the partition $\lambda$. If $\lambda$ and $\delta$ are two partitions we define the \textit{union} $\lambda \cup \delta$ as the following partition:
$$\lambda \cup \delta=(1^{m_1(\lambda)+m_1(\delta)},2^{m_2(\lambda)+m_2(\delta)},3^{m_3(\lambda)+m_3(\delta)},\ldots).$$
A partition is called \textit{proper} if it does not have any part equal to 1. The proper partition associated to a partition $\lambda$ is the partition $\bar{\lambda}:=\lambda \setminus (1^{m_1(\lambda)})=(2^{m_2(\lambda)},3^{m_3(\lambda)},\ldots).$ 
\subsection{Permutations and Coset type}\label{section 2.2} For a permutation $\omega$, we use the word notation $\omega_1\,\omega_2\, \cdots \, \omega_n$, where $\omega_i=\omega(i)$. The set $\mathcal{S}_n$ of all permutations of $[n]$ is a group for the composition called the \textit{symmetric group} of size $n$.

To each permutation $\omega$ of $2n$ we associate a graph $\Gamma(\omega)$ with $2n$ vertices located on a circle. Each vertex is labelled by two labels (\textit{exterior} and \textit{interior}). The exterior labels run through natural numbers from $1$ to $2n$ around the circle. The interior label of the vertex with exterior label $i$ is $\omega(i)$. We link the vertices with exterior (resp. interior) labels $2i-1$ and $2i$ by an exterior (resp. interior) edge. As every vertex has degree 2, the graph $\Gamma(\omega)$ is a disjoint union of cycles since exterior and interior edges alternate, all cycles have even lengths $2\lambda_1\geq 2\lambda_2\geq 2\lambda_3\geq \cdots$. The \textit{coset-type} of $\omega$ denoted by $ct(\omega)$ is the partition $(\lambda_1,\lambda_2,\lambda_3,\ldots)$ of $n$.
\begin{ex}\label{exemple du coset type} The graph $\Gamma(\omega)$ associated to the permutation $\omega=2\,4\,9\,3\,1\,10\,5\,8\,6\,7\in \mathcal{S}_{10}$ is drawn on Figure~\ref{fig:coset-type}.
\begin{figure}[htbp]
\begin{center}
\begin{tikzpicture}
\foreach \angle / \label in
{ 0/2, 36/4, 72/9, 108/3, 144/1, 180/10, 216/5,
252/8, 288/6, 324/7}
{
\node at (\angle:2cm) {\footnotesize $\bullet$};
\draw[red] (\angle:1.7cm) node{\textsf{\label}};
}
\foreach \angle / \label in
{ 0/1, 36/2, 72/3, 108/4, 144/5, 180/6, 216/7,
252/8, 288/9, 324/10}
{\draw (\angle:2.3cm) node{\textsf{\label}};}
\draw (0:2cm) .. controls + (5mm,0) and +(.5,-.1) .. (36:2cm);
\draw[red] (36:2cm) .. controls + (-.5,.1) and + (0,-.2) .. (108:2);
\draw (108:2cm) .. controls + (0,.2) and + (0,.2) .. (72:2);
\draw[red]	(72:2)		  .. controls + (0,-.2) and + (+.5,.1) .. (180:2);
\draw	(180:2)		  .. controls + (-5mm,0) and + (-0.5,0) .. (144:2);
\draw[red]	(144:2)		  .. controls + (1.2,0) and + (0,.2) .. (0:2);
\draw[red] (216:2cm) .. controls + (0,15mm) and +(-.5,-.1) .. (288:2cm);
\draw	(288:2cm)		    .. controls + (1,0.1) and + (0,-.2) .. (324:2);
\draw[red]	(324:2)		    .. controls + (-0.1,1.5) and + (1,0.1) .. (252:2);
\draw	(252:2)		    .. controls + (0,0) and + (0,-1) .. (216:2);
\end{tikzpicture}
\caption{The graph $\Gamma(\omega)$ from Example \ref{exemple du coset type}.}
\label{fig:coset-type}
\end{center}
\end{figure}
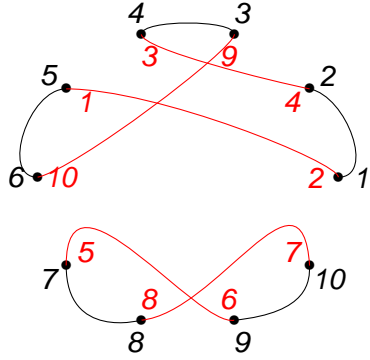
It has two cycles of length 6 and 4, so $ct(\omega)=(3,2)$.
\end{ex}
For every  $k\geq 1$, we set $\rho(k):=\lbrace 2k-1,2k\rbrace$. The \textit{hyperoctahedral group} $\mathcal{B}_n$ is the subgroup of $\mathcal{S}_{2n}$ of permutations $\omega$ such that, for every $1\leq k\leq n$, there exists $1\leq k'\leq n$ with $\omega(\rho(k))=\rho(k')$. In other words $\mathcal{B}_n=\lbrace \omega\in \mathcal{S}_{2n}~|~ct(\omega)=(1^n)\rbrace$. For example, $4\,3\,1\,2\,6\,5\in \mathcal{B}_3$.

A $\mathcal{B}_n$-\textit{double coset} of $\mathcal{S}_{2n}$ is the set $\mathcal{B}_nx\mathcal{B}_n=\lbrace bxb'~;~b,b'\in \mathcal{B}_n\rbrace$ for some $x\in \mathcal{S}_{2n}$. It is known, see \cite[page 401]{McDo}, that two permutations of $\mathcal{S}_{2n}$ are in the same $\mathcal{B}_n$-double coset if and only if they have the same coset-type. Thus, if $x\in \mathcal{S}_{2n}$ has coset-type $\lambda$, we have:
$$\mathcal{B}_nx\mathcal{B}_n=\lbrace y\in \mathcal{S}_{2n}\text{ such that } ct(y)=\lambda\rbrace.$$
\subsection{The Hecke algebra of $(\mathcal{S}_{2n},\mathcal{B}_n)$}
The \textit{symmetric group algebra of $n$}, denoted by $\mathbb{C}[\mathcal{S}_n]$, is the algebra over $\mathbb{C}$ linearly generated by all permutations of $[n]$. The group $\mathcal{B}_n\times \mathcal{B}_n$ acts on $\mathbb{C}([\mathcal{S}_{2n}])$ by the following action:
$(b,b')\cdot x=bxb'^{-1},$ called the $\mathcal{B}_n\times \mathcal{B}_n$-action. The \textit{Hecke algebra} of $(\mathcal{S}_{2n},\mathcal{B}_n)$, denoted by $\mathbb{C}[\mathcal{B}_n\setminus \mathcal{S}_{2n}/\mathcal{B}_n]$, is the sub-algebra of $\mathbb{C}[\mathcal{S}_{2n}]$ of elements invariant under the $\mathcal{B}_n\times \mathcal{B}_n$-action. Recall that $\mathcal{B}_n$-double cosets are indexed by partitions of $n$. Here, we rather index the basis by proper partitions of size less or equal to $n$, which are trivially in bijection with partitions of $n$. Therefore, the set 
$$\lbrace K_\lambda(n):\lambda\text{ is a proper partition with } |\lambda|\leq n\rbrace$$ 
forms a basis for $\mathbb{C}[\mathcal{B}_n\setminus \mathcal{S}_{2n}/\mathcal{B}_n]$, where $K_\lambda(n)$ is the sum of all permutations from $\mathcal{S}_{2n}$ with coset-type $\lambda \cup 1^{n-|\lambda|}$. So, for any two proper partitions $\lambda$ and $\delta$ with size at most $n$, there exist complex numbers $\alpha_{\lambda\delta}^{\rho}(n)$ such that:
\begin{equation}\label{equation 1}
K_{\lambda}(n)\cdot K_{\delta}(n)=\sum_{\rho\text{ proper partition } \atop {|\rho|\leq n}}\alpha_{\lambda\delta}^{\rho}(n)K_{\rho}(n).
\end{equation}

\subsection{Main result}
In this paper, we obtain a polynomiality property for the structure coefficients $\alpha_{\lambda\delta}^{\rho}(n)$ of the Hecke algebra of $(\mathcal{S}_{2n},\mathcal{B}_n)$. More precisely, we prove the following theorem. We will use the standard notation $(n)_{k}:=\frac{n!}{(n-k)!}=n(n-1)\cdots (n-k+1)$.
\begin{theoreme} \label{Theorem 2.1}
Let $\lambda$, $\delta$ and $\rho$ be three proper partitions. Than we have:
$$\alpha_{\lambda\delta}^{\rho}(n)=\left\{
\begin{array}{ll}
  2^nn!f_{\lambda\delta}^{\rho}(n) & \qquad \mathrm{if}\quad n\geq |\rho|,\\\\
  0 & \qquad \mathrm{if}\quad n< |\rho|, \\
 \end{array}
 \right.$$
 where $\displaystyle{  f_{\lambda\delta}^{\rho}(n)=\sum_{j=0}^{|\lambda|+|\delta|-|\rho|}a_j(n-|\rho|)_{j}}$ is a polynomial in $n$ and the $a_j$'s are non-negative rational numbers.
\end{theoreme}
\begin{ex}
Let us compute the structure coefficient $\alpha_{(2)(2)}^{\emptyset}(n)$. We have:
$$K_{(2)}(n)=\sum_{\omega\in \mathcal{S}_{2n} \atop {ct(\omega)=(2)\cup (1^{n-2})}}\omega.$$
To find the coefficient of $K_\emptyset(n)$ in $K_{(2)}(n)\cdot K_{(2)}(n)$, we fix a permutation with coset-type $(1^n)$, for example $\Id_{2n}$, and we look in how many ways we can obtain $\Id_{2n}$ as a product of two elements $\sigma\cdot \beta$ where $ct(\sigma)=ct(\beta)=(2,1^{n-2})$. Thus we are looking for the number of permutations $\sigma\in \mathcal{S}_{2n}$ such that $ct(\sigma)=ct(\sigma^{-1})=(2,1^{n-2})$. But, for any $\sigma\in \mathcal{S}_{2n}$ with $ct(\sigma)=(2,1^{n-2})$, its inverse has the same coset-type. Therefore $\alpha_{(2)(2)}^{\emptyset}(n)$ is the number of permutations of coset-type $(2,1^{n-2})$, which is $$\frac{(2^nn!)^2}{2^{n-1}(2 \cdot (n-2)!)}=n(n-1)2^nn!,$$ by \cite[page 402]{McDo}.
\end{ex}
\subsection{Major steps of the proof}\label{section 2.5}
The idea of the proof is to build a universal algebra $\mathcal{A}_\infty$ over $\mathbb{C}$ satisfying the following properties:
\begin{enumerate}[topsep=1pt, partopsep=1pt, itemsep=1pt, parsep=1pt]
\item For every $n\in \mathbb{N}^*$, there exists a morphism of algebras $\theta_n:\mathcal{A}_\infty \longrightarrow \mathbb{C}[\mathcal{B}_n\setminus \mathcal{S}_{2n}/\mathcal{B}_n].$
\item Every element $x$ in $\mathcal{A}_\infty$ is written in a unique way as an infinite linear combination of elements $T_\lambda$, indexed by partitions. This implies that, for any two partitions $\lambda$ and $\delta$, there exist non-negative rational numbers $b_{\lambda\delta}^\rho$ such that:
\begin{equation}\label{equation 2}
T_\lambda \pr T_\delta=\sum_{\rho\text{ partition}}b_{\lambda\delta}^\rho T_\rho.
\end{equation} 
\item The morphism $\theta_n$ sends $T_\lambda$ to a multiple of $K_{\bar{\lambda}}(n)$.
\end{enumerate}
To build $\mathcal{A}_\infty$, we use combinatorial objects called \textit{partial bijections.} For every $n\in \mathbb{N}^*$, we construct an algebra $\mathcal{A}_n$ using the set of partial bijections of size $n$. The algebra $\mathcal{A}_\infty$ is defined as the projective limit of this sequence $(\mathcal{A}_n)$.

The projection $p_n:\mathcal{A}_\infty\rightarrow \mathcal{A}_n$ involves coefficients which are polynomials in $n$. By defining the extension of a partial bijection of $n$ to the set $[2n]$, we construct a morphism from $\mathcal{A}_n$ to $\mathbb{C}[\mathcal{B}_n\setminus \mathcal{S}_{2n}/\mathcal{B}_n]$. Its coefficients involve the number $2^nn!$. It turns out that the morphism $\theta_n$ is the composition of those two morphisms: 
$$\xymatrix{
    \mathcal{A}_\infty \ar[dd]_{\theta_{n}} \ar[rd]^*[@]{\hbox to 0pt{\hss\txt{$p_n$}\hss}}& \\
     & \mathcal{A}_n \ar[dl] \\
     \mathbb{C}[\mathcal{B}_n\setminus \mathcal{S}_{2n}/\mathcal{B}_n]}$$
The final step consists of applying the chain of homomorphisms in the diagram above to equation (\ref{equation 2}).
\begin{rem}
This method is based on Ivanov and Kerov's one to get the polynomiality of the structure coefficients of the center of the symmetric group algebra (see \cite{IvanovKerov1999} for more details). Nevertheless, our construction is more complicated, mainly because a partial bijection does not have a unique trivial extension to a given set, see Definition \ref{definition 3.2}.
\end{rem}
\section{The partial bijection algebra} \label{section 3}
In this section we define the set of partial bijections of $n$. With this set, we build the algebras and homomorphisms that appear in the diagram above.
\subsection{Definition}
We start by defining partial bijections of $n$ and the partial bijection algebra. Then, we introduce the notion of trivial extension of a partial bijection of $n$ and we use it to build a homomorphism between the partial bijection algebra of $n$ and the symmetric group algebra of $2n$.

Let $\mathbb{N}^{*}$ denotes the set of positive integers. For $n\in \mathbb{N}^{*}$, we define $\mathbf{P}_{n}$ to be the following set: $$\mathbf{P}_{n}:=\lbrace \rho(k_1)\cup\cdots\cup \rho(k_i)~|~ 1 \leq i \leq n,~1\leq k_1<\cdots < k_i \leq n \rbrace.$$
\begin{definition}
A \textit{partial bijection} of $n$ is a triple $(\sigma, d ,d^{'})$ where $d,d^{'}\in \mathbf{P}_n$ and $\sigma:d\longrightarrow d^{'}$ is a bijection. The set $d$ is the domain of $(\sigma, d ,d^{'})$ while $d^{'}$ is its codomain. We denote by $Q_n$ the set of all partial bijections of $n$.
\end{definition}

For any positive integer $n$, let $R_n$ be the set of all one-to-one maps $f:d(f)\longrightarrow c(f)$ where $c(f),d(f)\subseteq [n].$ The set $R_n$ with the composition of maps is a monoid -- that is the composition is associative and $R_n$ has an identity element -- called the \textit{symmetric inverse semigroup}. With this composition, the set of partial bijections $Q_n$ forms a submonoid of $R_{2n}.$ It is known, see \cite{Solomon2002309}, that $R_{2n}$ is in bijection with the \textit{hook monoid} $\mathcal{R}_{2n}.$ It is important to notice that this obvious structure on $Q_n$ does not enter the picture in here. The useful product in this work will be defined later in this section.

It should be clear that
$$\displaystyle{| Q_n|=\sum_{k=0}^{n}\begin{pmatrix}
n\\
k
\end{pmatrix}^2 (2k)!}.$$
A permutation $\sigma$ of $2n$ can be written as $(\sigma,[2n],[2n])$, so the set $\mathcal{S}_{2n}$ can be considered as a subset of $Q_n$.\\
\begin{notation}For any partial bijection $\alpha$, we will use the convention that $\sigma$ (resp. $d$, $d^{'}$) is the first (resp. second, third) element of the triple defining $\alpha$. The same convention holds for $\widetilde{\alpha}$, $\alpha_i$, $\hat{\alpha}$ \ldots
\end{notation}
\begin{obs}
In the same way as in Section \ref{section 2.2}, we can associate to each partial bijection $\alpha$ of $n$ a graph $\Gamma (\alpha)$ with $|d|$ vertices placed on a circle. The exterior (resp. interior) labels are the elements of the set $d$ (resp. $d^{'}$). Since the sets $d$ and $d^{'}$ are in $\mathbf{P}_n$, we can link $2i$ with $2i-1$ as in the case $d=d^{'}=[2n]$. So, the definition of coset-type extends naturally to partial bijections. We denote by $ct(\alpha)$ or $ct(\sigma)$ the coset-type of a partial bijection $\alpha$. 
\end{obs}
\begin{ex}\label{coset type partial bijection} Let $\alpha=(\sigma,d,d')$ be the partial bijection of $16$ where $d=\lbrace 3,4,5,6,9,10,11,12,13,14\rbrace$, $d'=\lbrace 1,2,3,4,7,8,9,10,15,16\rbrace$ and $\sigma$ is given by the following two lines notation: 
$$\sigma=\begin{matrix}
3&4&5&6&9&10&11&12&13&14\\
9&16&1&15&10&2&4&8&3&7
\end{matrix},$$
which means that $\sigma(3)=9$, $\sigma(4)=1$ and so on. The graph $\Gamma(\alpha)$ is drawn on Figure~\ref{fig:coset-type partial bijection}.
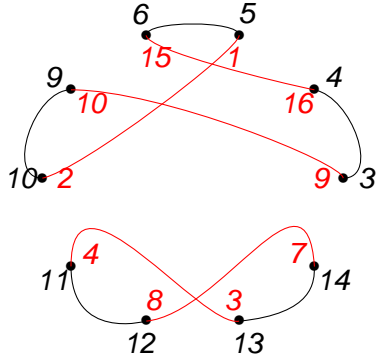
\begin{figure}[htbp]
\begin{center}
\begin{tikzpicture}
\foreach \angle / \label in
{ 0/9, 36/16, 72/1, 108/15, 144/10, 180/2, 216/4,
252/8, 288/3, 324/7}
{
\node at (\angle:2cm) {\footnotesize $\bullet$};
\draw[red] (\angle:1.7cm) node{\textsf{\label}};
}
\foreach \angle / \label in
{ 0/3, 36/4, 72/5, 108/6, 144/9, 180/10, 216/11,
252/12, 288/13, 324/14}
{\draw (\angle:2.3cm) node{\textsf{\label}};}
\draw (0:2cm) .. controls + (5mm,0) and +(.5,-.1) .. (36:2cm);
\draw[red] (36:2cm) .. controls + (-.5,.1) and + (0,-.2) .. (108:2);
\draw (108:2cm) .. controls + (0,.2) and + (0,.2) .. (72:2);
\draw[red]	(72:2)		  .. controls + (0,-.2) and + (+.5,.1) .. (180:2);
\draw	(180:2)		  .. controls + (-5mm,0) and + (-0.5,0) .. (144:2);
\draw[red]	(144:2)		  .. controls + (1.2,0) and + (0,.2) .. (0:2);
\draw[red] (216:2cm) .. controls + (0,15mm) and +(-.5,-.1) .. (288:2cm);
\draw	(288:2cm)		    .. controls + (1,0.1) and + (0,-.2) .. (324:2);
\draw[red]	(324:2)		    .. controls + (-0.1,1.5) and + (1,0.1) .. (252:2);
\draw	(252:2)		    .. controls + (0,0) and + (0,-1) .. (216:2);
\end{tikzpicture}
\caption{The graph $\Gamma(\alpha)$.}
\label{fig:coset-type partial bijection}
\end{center}
\end{figure}
It has two cycles of length 6 and 4, so $ct(\alpha)=(3,2)$.
\end{ex}
\begin{definition}\label{definition 3.2}
Let $(\sigma,d,d^{'})$ and $(\widetilde{\sigma},\widetilde{d},\widetilde{d^{'}})$ be two partial bijections of $n$. We say that $(\widetilde{\sigma},\widetilde{d},\widetilde{d^{'}})$ is a \textit{trivial extension} of $(\sigma,d,d^{'})$ if:
$$d\subseteq \widetilde{d},~ \widetilde{\sigma}_{|_{d}}=\sigma ~ \text{and} ~ ct(\widetilde{\sigma})=ct(\sigma)\cup \Big(1^\frac{{|\widetilde{d}\setminus d|}}{2}\Big).$$
We denote by $P_\alpha(n)$ the set of all trivial extensions of $\alpha$ in $Q_n$.
\end{definition}
\begin{ex}
Let $\alpha$ be the partial bijection of $16$ given in Example \ref{coset type partial bijection}. Let $\widetilde{d}=d\cup \lbrace 1,2,15,16\rbrace$, $\widetilde{d^{'}}=d'\cup \lbrace 5,6,11,12\rbrace$ and consider the following bijection: $$\widetilde{\sigma}=
\setcounter{MaxMatrixCols}{16}
\begin{matrix}
\bf{1}&\bf{2}&3&4&5&6&9&10&11&12&13&14&\bf{15}&\bf{16}\\
\bf{12}&\bf{11}&9&16&1&15&10&2&4&8&3&7&\bf{5}&\bf{6}
\end{matrix}.$$
Then, $\widetilde{\alpha}=(\widetilde{\sigma},\widetilde{d},\widetilde{d^{'}})$ is a trivial extension of $\alpha$. In the same way $\hat{\alpha}=(\hat{\sigma},\hat{d},\hat{d^{'}})$, where $\hat{d}=\hat{d^{'}}=[16]$ and 
$$\hat{\sigma}=
\setcounter{MaxMatrixCols}{16}
\begin{matrix}
\bf{1}&\bf{2}&3&4&5&6&\bf{7}&\bf{8}&9&10&11&12&13&14&\bf{15}&\bf{16}\\
\bf{13}&\bf{14}&9&16&1&15&\bf{12}&\bf{11}&10&2&4&8&3&7&\bf{6}&\bf{5}
\end{matrix},$$
is also a trivial extension of $\alpha.$
\end{ex}
\begin{lem}\label{Lemma 3.1}
Let $\alpha$ be a partial bijection of $n$ and $X$ an element of $\mathbf{P}_n$ such that $d\subseteq X$. The number of trivial extensions $\widetilde{\alpha}$ of $\alpha$ such that $\widetilde{d}=X$ is
$$(2n-|d|)\cdot (2n-|d|-2)\cdots (2n-|d|-|X\setminus d|+2)=2^{\frac{|X\setminus d|}{2}}\Big(n-\frac{|d|}{2}\Big)_{\frac{|X\setminus d|}{2}}.$$ 
We have the same formula for the number of trivial extensions $\widetilde{\alpha}$ such that $\widetilde{d'}=X$.
\end{lem}
\begin{proof} Straightforward by induction.
\end{proof}
Consider $\mathcal{D}_n=\mathbb{C}[Q_n]$ the vector space with basis $Q_n$. We want to endow it with an algebra structure. Let $\alpha_1$ and $\alpha_2$ be two partial bijections. If $d_1=d'_2$, we can compose $\alpha_1$ and $\alpha_2$ and we define $\alpha_1\pr \alpha_2=\alpha_1\circ \alpha_2=(\sigma_1\circ \sigma_2,d_2,d'_1)$. Otherwise, we need to extend $\alpha_1$ and $\alpha_2$ to partial bijections $\widetilde{\alpha_1}$ and $\widetilde{\alpha_2}$ such that $\widetilde{d_1}=\widetilde{d'_2}$.
Since there exist several trivial extensions of $\alpha_1$ and $\alpha_2$, a natural choice is to take the average of the composition of all possible trivial extensions. Let $E_{\alpha_1}^{\alpha_2}(n)$ be the following set:
\begin{eqnarray*}
E_{\alpha_1}^{\alpha_2}(n)&:=&\lbrace(\widetilde{\alpha_1},\widetilde{\alpha_2})\in P_{\alpha_1}(n)\times P_{\alpha_2}(n) \text{ such that }\widetilde{d_1}=\widetilde{d_2'}=d_1\cup d_2'\rbrace.
\end{eqnarray*}
Elements of $E_{\alpha_1}^{\alpha_2}(n)$ are schematically represented on Figure \ref{fig:composition}.
\begin{figure}[htbp]
\begin{center}
\begin{tikzpicture}[thick,fill opacity=0.5]
\draw (0,0) ellipse (.5cm and .8cm)++(2,0) ellipse (.5cm and .8cm);
\draw[->] (0.5,0) to node [sloped,above] {$\sigma_2$} (1.5,0);
\draw[->,dashed] (0.5,-1)[] to node [sloped,above] {$\widetilde{\sigma_2}$} (1.5,-1);
\draw (2,-1) ellipse (.5cm and .8cm);
\draw (4,-1) ellipse (.5cm and .8cm);
\draw[->] (2.5,-1) to node [sloped,above] {$\sigma_1$} (3.5,-1);
\draw[->, dashed] (2.5,0)[] to node [sloped,above] {$\widetilde{\sigma_1}$} (3.5,0);
\draw[dashed] (-0.4,-0.5) arc (140:400: 0.5cm and 0.8cm);
\draw[dashed] (4.4,-0.4) arc (-40:220: 0.5cm and .8cm);
\draw (0,0) node {$d_2$};
\draw (2,0) node {$d_2'$};
\draw (2,-1) node {$d_1$};
\draw (4,-1) node {$d_1'$};
\end{tikzpicture}
\caption{Schematic representation of elements of $E_{\alpha_1}^{\alpha_2}(n)$.}
\label{fig:composition}
\end{center}
\end{figure}
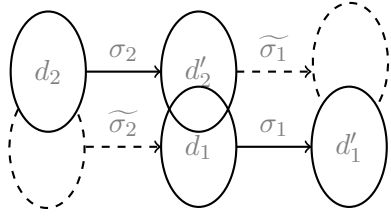 

$\textit{ Note on a convention for figure.}$ Throughout the paper, we will use the following conventions on figures that represent elements of some sets.
\begin{enumerate}
\item[-] The data defining the set (so fixed when we go from an element of the set to another) is drawn using plain shapes. 
\item[-] The element of the set is drawn using dashed shapes.
\end{enumerate}

\begin{ex}\label{Ex E}
To simplify, we will unify a partial bijection $\alpha$ with its bijection $\sigma$ using the two line notations. The first line will represent the set $d$ while the second will represent $d'$. Consider the two partial bijections $\alpha_1$ and $\alpha_2$ of $3$ given below,
$$\alpha_1=\begin{matrix}
1&2&5&6\\
3&2&1&4
\end{matrix}~~\text{ and }~~ \alpha_2=\begin{matrix}
3&4&5&6\\
5&6&3&4
\end{matrix}.$$
Then, $E_{\alpha_1}^{\alpha_2}(3)$ is the set of the following four elements. 
$$\left(\begin{matrix}
1&2&\bf{3}&\bf{4}&5&6\\
3&2&\bf{5}&\bf{6}&1&4
\end{matrix}~~,~~\begin{matrix}
\bf{1}&\bf{2}&3&4&5&6\\
\bf{1}&\bf{2}&5&6&3&4
\end{matrix}\right),\left(\begin{matrix}
1&2&\bf{3}&\bf{4}&5&6\\
3&2&\bf{6}&\bf{5}&1&4
\end{matrix}~~,~~\begin{matrix}
\bf{1}&\bf{2}&3&4&5&6\\
\bf{1}&\bf{2}&5&6&3&4
\end{matrix}\right),$$
$$\left(\begin{matrix}
1&2&\bf{3}&\bf{4}&5&6\\
3&2&\bf{5}&\bf{6}&1&4
\end{matrix}~~,~~\begin{matrix}
\bf{1}&\bf{2}&3&4&5&6\\
\bf{2}&\bf{1}&5&6&3&4
\end{matrix}\right)\text{ and }\left(\begin{matrix}
1&2&\bf{3}&\bf{4}&5&6\\
3&2&\bf{6}&\bf{5}&1&4
\end{matrix}~~,~~\begin{matrix}
\bf{1}&\bf{2}&3&4&5&6\\
\bf{2}&\bf{1}&5&6&3&4
\end{matrix}\right).$$
\end{ex}

We define the product of $\alpha_1$ and $\alpha_2$ as follows:
\begin{equation}\label{equation 3}
\alpha_1\pr \alpha_2:=\frac{1}{| E_{\alpha_1}^{\alpha_2}(n)|}\sum_{(\widetilde{\alpha_1},\widetilde{\alpha_2}) \in E_{\alpha_1}^{\alpha_2}(n)} \widetilde{\alpha_1}\circ \widetilde{\alpha_2}.
\end{equation}
By Lemma \ref{Lemma 3.1}, we have: 
\begin{equation}\label{cardinal de E}
| E_{\alpha_1}^{\alpha_2}(n)|=2^{\frac{| d'_2\setminus d_1|}{2}+\frac{| d_1\setminus d'_2|}{2}}.(n-\frac{| d'_1|}{2})_{(\frac{| d'_2\setminus d_1|}{2})}.(n-\frac{| d_2|}{2})_{(\frac{| d_1\setminus d'_2|}{2})}.
\end{equation}
\begin{prop}
The product $\pr$ is associative. In other words, $\mathcal{D}_n$ is a (non-unital) algebra.
\end{prop}
\begin{proof}
Postponed to the next section.
\end{proof}
We will illustrate the associativity by a simple example given below. 
\begin{ex}
Let $\alpha_1$ and $\alpha_2$ be the two partial bijections of $6$ given in Example \ref{Ex E}, using the set $E_{\alpha_1}^{\alpha_2}(3)$, we have:
$$\alpha_1\pr \alpha_2=\frac{1}{4}\left[ \begin{matrix}
1&2&3&4&5&6\\
3&2&1&4&5&6
\end{matrix}~+~\begin{matrix}
1&2&3&4&5&6\\
3&2&1&4&6&5
\end{matrix}~+~\begin{matrix}
1&2&3&4&5&6\\
2&3&1&4&5&6
\end{matrix}~+~\begin{matrix}
1&2&3&4&5&6\\
2&3&1&4&6&5
\end{matrix}\right].$$
Take the partial bijection $\alpha_3$ of $6$ defined by: 
$$\alpha_3=\begin{matrix}
1&2&3&4\\
5&1&6&2
\end{matrix}.$$
In the same way we can verify that:
$$\alpha_2\pr \alpha_3=\frac{1}{4}\left[ \begin{matrix}
1&2&3&4&5&6\\
3&1&4&2&5&6
\end{matrix}~+~\begin{matrix}
1&2&3&4&5&6\\
3&1&4&2&6&5
\end{matrix}~+~\begin{matrix}
1&2&3&4&5&6\\
3&2&4&1&5&6
\end{matrix}~+~\begin{matrix}
1&2&3&4&5&6\\
3&2&4&1&6&5
\end{matrix}\right],$$
and that:
$$\alpha_1\pr (\alpha_2\pr \alpha_3)=(\alpha_1\pr \alpha_2)\pr \alpha_3=\frac{1}{8}\left[ \begin{matrix}
1&2&3&4&5&6\\
5&3&6&2&1&4
\end{matrix}~+~\begin{matrix}
1&2&3&4&5&6\\
5&3&6&2&4&1
\end{matrix}~+~\begin{matrix}
1&2&3&4&5&6\\
6&3&5&2&1&4
\end{matrix}\right]$$
$$+\frac{1}{8}\left[\begin{matrix}
1&2&3&4&5&6\\
6&3&5&2&4&1
\end{matrix}~+~\begin{matrix}
1&2&3&4&5&6\\
5&2&6&3&1&4
\end{matrix}~+~\begin{matrix}
1&2&3&4&5&6\\
5&2&6&3&4&1
\end{matrix}\right]$$
$$+\frac{1}{8}\left[\begin{matrix}
1&2&3&4&5&6\\
6&2&5&3&1&4
\end{matrix}~+~\begin{matrix}
1&2&3&4&5&6\\
6&2&5&3&4&1
\end{matrix}\right].$$
\end{ex}
\begin{prop}\label{proposition 3.2}
The following function extends to an algebra homomorphism between $\mathbb{C}[Q_n]$ and $\mathbb{C}[\mathcal{S}_{2n}].$
$$\begin{array}{ccccc}
\psi_n & : & \mathbb{C}[Q_n] & \longrightarrow & \mathbb{C}[\mathcal{S}_{2n}] \\
& & \alpha & \mapsto & \frac{1}{2^{n-\frac{|d|}{2}}(n-\frac{|d|}{2})!}\displaystyle{\sum_{\hat{\alpha}\in \mathcal{S}_{2n}\cap P_\alpha(n) }}\hat{\sigma}\\
\end{array}.$$
\end{prop}
\begin{proof}
Let $\alpha_1$ and $\alpha_2$ be two basis elements of $\mathbb{C}[Q_n]$. We refer to Figure \ref{fig:composition} and denote\linebreak $2b=|d_2|=|d'_2|, 2c=|d_1|=|d'_1|$ and $2e=|d'_2\cap d_1|$. We first prove that:
\begin{multline}\label{T}
\sum_{\widehat{\alpha_1}\in \mathcal{S}_{2n}\cap P_{\alpha_1}(n) } \quad
\sum_{\widehat{\alpha_2}\in \mathcal{S}_{2n}\cap P_{\alpha_2}(n) }
\widehat{\sigma_1}\circ \widehat{\sigma_2} \\
=2^{n-(b+c-e)}(n-(b+c-e))!\sum_{(\widetilde{\alpha_1},\widetilde{\alpha_2})\in E^{\alpha_1}_{\alpha_2}(n)} \quad \sum_{\widehat{\widetilde{\alpha_1}\circ \widetilde{\alpha_2}}\in \mathcal{S}_{2n}\cap P_{\widetilde{\alpha_1}\circ \widetilde{\alpha_2}}(n)}\widehat{\widetilde{\sigma_1}\circ \widetilde{\sigma_2}}.
\end{multline}
We fix $(\widetilde{\alpha_1},\widetilde{\alpha_2})\in E_{\alpha_1}^{\alpha_2}(n)$ and $\omega\in \mathcal{S}_{2n}\cap P_{\widetilde{\alpha_1}\circ \widetilde{\alpha_2}}(n)$, \textit{i.e.}: $$\omega_{|_{\widetilde{d_2}}}=\widetilde{\sigma_1}\circ \widetilde{\sigma_2} \text{ and } ct(\omega)=ct(\widetilde{\sigma_1}\circ\widetilde{\sigma_2})\cup (1^{(n-(b+c-e))}).$$
We look for the number of permutations $\widehat{\sigma_1}$ and $\widehat{\sigma_2}$ in $\mathcal{S}_{2n}\cap P_{\alpha_1}(n)$ and $\mathcal{S}_{2n}\cap P_{\alpha_2}(n)$ such that $\widehat{\sigma_1}\circ \widehat{\sigma_2}=\omega$. In this equation, $\widehat{\sigma_2}$ determines $\widehat{\sigma_1}$. But the condition $\omega_{|_{\widetilde{d_2}}}=\widetilde{\sigma_1}\circ \widetilde{\sigma_2}$ gives the values of $\widehat{\sigma_2}$ on $\widetilde{d_2}$ ($\widehat{\sigma_2}(x)=\sigma_2(x)$ if $x\in d_2$ and $\widehat{\sigma_2}(x)= \sigma_1^{-1}(\omega(x))$ if $x\in \widetilde{d_2}\setminus d_2$). Thus, the number of ways to choose $\widehat{\sigma_2}$ is the number of ways to extend trivially $\widetilde{\sigma_2}$ to a permutation of $2n$, which is $2^{n-(b+c-e)}(n-(b+c-e))!$ by Lemma \ref{Lemma 3.1}. This proves equation (\ref{T}).

Now we have:
\begin{multline}\label{N}
\psi_n(\alpha_1)\psi_n(\alpha_2)=\frac{1}{2^{2n-b-c}(n-c)!(n-b)!}\displaystyle{\sum_{\widehat{\alpha_1}\in \mathcal{S}_{2n}\cap P_{\alpha_1}(n)}}~\displaystyle{\sum_{\widehat{\alpha_2}\in \mathcal{S}_{2n}\cap P_{\alpha_2}(n)}}\widehat{\sigma_1}\circ \widehat{\sigma_2}\\
=\frac{(n-b-c+e)!}{2^{n-e}(n-c)!(n-b)!}\sum_{(\widetilde{\alpha_1},\widetilde{\alpha_2})\in E^{\alpha_1}_{\alpha_2}(n)}~\sum_{\widehat{\widetilde{\alpha_1}\circ \widetilde{\alpha_2}}\in \mathcal{S}_{2n}\cap P_{\widetilde{\alpha_1}\circ \widetilde{\alpha_2}}(n)}\widehat{\widetilde{\sigma_1}\circ \widetilde{\sigma_2}}.~~~~~~~~~~~
\end{multline}
On the other hand:
$$\psi_n(\alpha_1\pr \alpha_2)=\frac{1}{2^{b+c-2e}(n-c)_{(b-e)}(n-b)_{(c-e)}}\sum_{(\widetilde{\alpha_1},\widetilde{\alpha_2})\in E^{\alpha_1}_{\alpha_2}(n)}\psi_n\big((\widetilde{\sigma_1}\circ \widetilde{\sigma_2},\widetilde{d_2},\widetilde{d_1'})\big).$$
But
$$\psi_n\big((\widetilde{\sigma_1}\circ \widetilde{\sigma_2},\widetilde{d_2},\widetilde{d_1'})\big)=\frac{1}{2^{n-(b+c-e)}(n-(b+c-e))!}\sum_{\widehat{\widetilde{\alpha_1}\circ \widetilde{\alpha_2}}\in \mathcal{S}_{2n}\cap P_{\widetilde{\alpha_1}\circ \widetilde{\alpha_2}}(n)}\widehat{\widetilde{\sigma_1}\circ \widetilde{\sigma_2}}.$$
Thus,
\begin{equation}\label{M}
\psi_n(\alpha_1\pr \alpha_2)=\frac{(n-b-c+e))!}{2^{n-e}(n-c)!(n-b)!}\sum_{(\widetilde{\alpha_1},\widetilde{\alpha_2})\in E^{\alpha_1}_{\alpha_2}(n)}~~\sum_{\widehat{\widetilde{\alpha_1}\circ \widetilde{\alpha_2}}\in \mathcal{S}_{2n}\cap P_{\widetilde{\alpha_1}\circ \widetilde{\alpha_2}}(n)}\widehat{\widetilde{\sigma_1}\circ \widetilde{\sigma_2}}.
\end{equation}
Comparing equations (\ref{N}) and (\ref{M}), we see that for any two partial bijections $\alpha_1$ and $\alpha_2$ of $n$, we have $\psi_n(\alpha_1 \pr \alpha_2)=\psi_n(\alpha_1)\psi_n(\alpha_2)$. In other words, $\psi_n$ is a homomorphism of algebras.
\end{proof}
\subsection{Proof of the associativity of $\pr$}
Let $\alpha_1,\alpha_2$ and $\alpha_3$ be three elements of $Q_n$. By definition of the product we have:
\begin{equation*}
(\alpha_1\pr \alpha_2)\pr \alpha_3=
\frac{1}{| E_{\alpha_1}^{\alpha_2}(n)|}\sum_{(\widetilde{\alpha_1},\widetilde{\alpha_2})\in E_{\alpha_1}^{\alpha_2}(n)}\frac{1}{| E_{\widetilde{\alpha_1}\circ \widetilde{\alpha_2}}^{\alpha_3}(n)|}\sum_{(\widetilde{\widetilde{\alpha_1}\circ \widetilde{\alpha_2}},\widetilde{\alpha_3})\in E_{\widetilde{\alpha_1}\circ \widetilde{\alpha_2}}^{\alpha_3}(n)}(\widetilde{\widetilde{\sigma_1}\circ \widetilde{\sigma_2}}\circ \widetilde{\sigma_3}, \widetilde{d_3},\widetilde{\widetilde{d^{'}_1}}),
\end{equation*}
\begin{equation*} \alpha_1\pr (\alpha_2\pr \alpha_3)=
\frac{1}{| E_{\alpha_2}^{\alpha_3}(n)|}\sum_{(\widetilde{\alpha_2},\widetilde{\alpha_3})\in E_{\alpha_2}^{\alpha_3}(n)}\frac{1}{| E_{\alpha_1}^{\widetilde{\alpha_2}\circ \widetilde{\alpha_3}}(n)|}\sum_{(\widetilde{\alpha_1},\widetilde{\widetilde{\alpha_2}\circ \widetilde{\alpha_3}})\in E_{\alpha_1}^{\widetilde{\alpha_2}\circ \widetilde{\alpha_3}}(n)}(\widetilde{\sigma_1}\circ \widetilde{\widetilde{\sigma_2}\circ \widetilde{\sigma_3}}, \widetilde{\widetilde{d_3}},\widetilde{d^{'}_1}).
\end{equation*}
We consider now the following sets, indexing the double sums in the equations above:
$$X_1=\lbrace \big((\widetilde{\alpha_1},\widetilde{\alpha_2}),(\widetilde{\widetilde{\alpha_1}\circ \widetilde{\alpha_2}},\widetilde{\alpha_3})\big) \text{ such that }(\widetilde{\alpha_1},\widetilde{\alpha_2})\in E_{\alpha_1}^{\alpha_2}(n)\text{ and } (\widetilde{\widetilde{\alpha_1}\circ \widetilde{\alpha_2}},\widetilde{\alpha_3})\in E_{\widetilde{\alpha_1}\circ \widetilde{\alpha_2}}^{\alpha_3}(n)\rbrace,$$
and
$$X_2=\lbrace \big((\widetilde{\alpha_2},\widetilde{\alpha_3}),(\widetilde{\alpha_1},\widetilde{\widetilde{\alpha_2}\circ \widetilde{\alpha_3}})\big)\text{ such that } (\widetilde{\alpha_2},\widetilde{\alpha_3})\in E_{\alpha_2}^{\alpha_3}(n)\text{ and } (\widetilde{\alpha_1},\widetilde{\widetilde{\alpha_2}\circ \widetilde{\alpha_3}})\in E_{\alpha_1}^{\widetilde{\alpha_2}\circ \widetilde{\alpha_3}}(n)\rbrace.$$
With the notation $X_1$, the product $(\alpha_1\pr \alpha_2)\pr \alpha_3$ can be written as follows:
\begin{equation}\label{produit 1}
(\alpha_1\pr \alpha_2)\pr \alpha_3=
\sum_{\big((\widetilde{\alpha_1},\widetilde{\alpha_2}),(\widetilde{\widetilde{\alpha_1}\circ \widetilde{\alpha_2}},\widetilde{\alpha_3})\big)\in X_1}\frac{1}{| E_{\alpha_1}^{\alpha_2}(n)|\cdot| E_{\widetilde{\alpha_1}\circ \widetilde{\alpha_2}}^{\alpha_3}(n)|}(\widetilde{\widetilde{\sigma_1}\circ \widetilde{\sigma_2}}\circ \widetilde{\sigma_3}, \widetilde{d_3},\widetilde{\widetilde{d^{'}_1}}).
\end{equation}
Schematically, the elements in $X_1$ are represented on Figure \ref{fig:associativity 1}.
\begin{figure}[htbp]
\begin{center}
\begin{tikzpicture}[thick,fill opacity=0.5]
\draw (1,0) ellipse (.7cm and 1cm) ++ (3,0) ellipse (.7cm and 1cm);
\draw (4,-1.5) ellipse (.7cm and 1cm)  ++(3,0) ellipse (.7cm and 1cm);
\draw (7,-3) ellipse (.7cm and 1cm)  ++(3,0) ellipse (.7cm and 1cm);
\draw[dashed] (3.6,-2.3) arc (115:415: 0.72cm and 0.9cm);
\draw[dashed] (10.4,-2.1) arc (310:600: 0.72cm and 0.9cm);
\draw[->] (1.7,0) to node [sloped,above] {$\sigma_3$} (3.3,0);
\draw[->] (4.7,-1.5) to node [sloped,above] {$\sigma_2$} (6.3,-1.5);
\draw[->] (7.7,-3) to node [sloped,above] {$\sigma_1$} (9.2,-3);
\draw[->,dashed] (1.7,-2.2) to node [sloped,above] {$\widetilde{\sigma_3}$} (3.3,-2.2);
\draw[->,dashed] (4.7,-3) to node [sloped,above] {$\widetilde{\sigma_2}$} (6.3,-3);
\draw[->,dashed] (7.7,-1.5) to node [sloped,above] {$\widetilde{\sigma_1}$} (9.2,-1.5);
\draw[snake=brace] (11,-0.5) -- (11,-4) node[midway,right] {$\widetilde{d_1^{'}}$};
\draw[->,dashed] (4.7,0) to node [sloped,above] {$\widetilde{\widetilde{\sigma_1}\circ \widetilde{\sigma_2}}$} (9.2,0);
\draw[snake=brace, mirror snake] (3.2,-4.2) -- (4.5,-4.2) node[midway,below] {$\widetilde{d'_3}$};
\draw[dashed] (0.6,-0.75) arc (120:415: 0.72cm and 1.75cm);
\draw[snake=brace, mirror snake] (-0.2,1) -- (-0.2,-4) node[midway,left] {$\widetilde{d_3}$};
\draw[dashed] (10.4,-0.7) arc (310:600: 0.72cm and 0.9cm);
\draw[snake=brace] (12,1) -- (12,-4) node[midway,right] {$\widetilde{\widetilde{d_1^{'}}}$};
\draw (1,0) node  {$d_3$};
\draw (4,0) node  {$d_3^{'}$};
\draw (4,-1.5) node  {$d_2$};
\draw (7,-1.5) node  {$d_2^{'}$};
\draw (7,-3) node  {$d_1$};
\draw (10,-3) node  {$d_1^{'}$};
\end{tikzpicture}
\caption{Schematic representation of elements in $X_1$.}
\label{fig:associativity 1}
\end{center}
\end{figure}\\
In the same way, with the notation $X_2$, the product $\alpha_1\pr (\alpha_2\pr \alpha_3)$ can be written in the following way:
\begin{equation}\label{produit 2}
\alpha_1\pr (\alpha_2\pr \alpha_3)=
\sum_{\big((\widetilde{\alpha_2},\widetilde{\alpha_3}),(\widetilde{\alpha_1},\widetilde{\widetilde{\alpha_2}\circ \widetilde{\alpha_3}})\big) \in X_2}\frac{1}{| E_{\alpha_2}^{\alpha_3}(n)|\cdot| E_{\alpha_1}^{\widetilde{\alpha_2}\circ \widetilde{\alpha_3}}(n)|}(\widetilde{\sigma_1}\circ \widetilde{\widetilde{\sigma_2}\circ \widetilde{\sigma_3}}, \widetilde{\widetilde{d_3}},\widetilde{d^{'}_1}).
\end{equation}
Schematically, the elements in $X_2$ are represented on Figure \ref{fig:associativity 2}.
\begin{figure}[htbp]
\begin{center}
\begin{tikzpicture}[thick,fill opacity=0.5]
\draw (1,0) ellipse (.7cm and 1cm) ++ (3,0) ellipse (.7cm and 1cm);
\draw (4,-1.5) ellipse (.7cm and 1cm)  ++(3,0) ellipse (.7cm and 1cm);
\draw (7,-3) ellipse (.7cm and 1cm)  ++(3,0) ellipse (.7cm and 1cm);
\draw[dashed] (0.6,-0.8) arc (120:415: 0.70cm and 0.9cm);
\draw[dashed] (7.4,-0.7) arc (310:600: 0.72cm and 0.9cm);
\draw[->] (1.7,0) to node [sloped,above] {$\sigma_3$} (3.3,0);
\draw[->] (4.7,-1.5) to node [sloped,above] {$\sigma_2$} (6.3,-1.5);
\draw[->] (7.7,-3) to node [sloped,above] {$\sigma_1$} (9.3,-3);
\draw[->,dashed] (1.7,-1.5) to node [sloped,above] {$\widetilde{\sigma_3}$} (3.3,-1.5);
\draw[->,dashed] (4.7,0) to node [sloped,above] {$\widetilde{\sigma_2}$} (6.2,0);
\draw[->,dashed] (7.7,-0.8) to node [sloped,above] {$\widetilde{\sigma_1}$} (9.2,-0.8);
\draw[->,dashed] (1.7,-3) to node [sloped,below] {$\widetilde{\widetilde{\sigma_2}\circ\widetilde{\sigma_3}}$} (6.2,-3);
\draw[snake=brace] (11,1) -- (11,-4) node[midway,right] {$\widetilde{d_1^{'}}$};
\draw[snake=brace, snake] (6.2,1) -- (7.6,1) node[midway,above] {$\widetilde{d_1}$};
\draw[dashed] (0.6,-2.3) arc (115:415: 0.70cm and 0.9cm);
\draw[snake=brace, mirror snake] (0.1,1) -- (0.1,-2.5) node[midway,left] {$\widetilde{d_3}$};
\draw[snake=brace, mirror snake] (-0.5,1) -- (-0.5,-4) node[midway,left] {$\widetilde{\widetilde{d_3}}$};
\draw[dashed] (10.35,-2.1) arc (310:600: 0.72cm and 1.8cm);
\draw (1,0) node  {$d_3$};
\draw (4,0) node  {$d_3^{'}$};
\draw (4,-1.5) node  {$d_2$};
\draw (7,-1.5) node  {$d_2^{'}$};
\draw (7,-3) node  {$d_1$};
\draw (10,-3) node  {$d_1^{'}$};
\end{tikzpicture}
\caption{Schematic representation of elements of $X_2$.}
\label{fig:associativity 2}
\end{center}
\end{figure}

To prove the associativity of the product, we build a set $X$ and two surjective functions $\phi_1:X\longrightarrow X_1$ and $\phi_2:X\longrightarrow X_2$ in order to write both sums (equations \eqref{produit 1} and \eqref{produit 2}) as sums over the same set $X$. Let $X$ be the set of elements 
$$\big(\epsilon_1=(\tau_1,\delta_0,\delta_1),\epsilon_2=(\tau_2,\delta_1,\delta_2),\epsilon_3=(\tau_3,\delta_2,\delta_3)\big)\in P_{\alpha_1}(n)\times P_{\alpha_2}(n)\times P_{\alpha_3}(n),$$ satisfying the following properties:
\begin{enumerate}[label=\roman*)]
\item $d_3^{'}\cup d_2\subseteq \delta_2$.
\item $d_2^{'}\cup d_1\subseteq \delta_1$.
\item \label{delta2} $\delta_{2}=d'_{3}\cup \tau_{2}^{-1}(d_1\cup d'_2)$.
\end{enumerate}
The elements of this set are schematically represented on Figure \ref{fig:associativity 3}. Note that \eqref{delta2} is a minimality condition. We will see in the proof of Lemma \ref{surjection de phi} below why it is useful.
\begin{figure}[htbp]
\begin{center}
\begin{tikzpicture}[thick,fill opacity=0.5]
\draw (1,0) ellipse (.7cm and 1cm) ++ (3,0) ellipse (.7cm and 1cm);
\draw (4,-1.5) ellipse (.7cm and 1cm)  ++(3,0) ellipse (.7cm and 1cm);
\draw (7,-3) ellipse (.7cm and 1cm)  ++(3,0) ellipse (.7cm and 1cm);
\draw[dashed] (7.4,-0.7) arc (310:600: 0.72cm and 0.9cm);
\draw[dashed] (3.6,-2.3) arc (115:415: 0.72cm and 0.9cm);
\draw[dashed] (0.6,-0.75) arc (120:415: 0.72cm and 1.75cm);
\draw[dashed] (10.35,-2.1) arc (310:600: 0.72cm and 1.8cm);
\draw[->] (1.7,0) to node [sloped,above] {$\tau_{3{_|{_{d_3}}}}$$=$$\sigma_3$} (3.3,0);
\draw[->] (4.7,-1.5) to node [sloped,above] {$\tau_{2{_|{_{d_2}}}}$$=$$\sigma_2$} (6.3,-1.5);
\draw[->] (7.7,-3) to node [sloped,above] {$\tau_{1{_|{_{d_1}}}}$$=$$\sigma_1$} (9.3,-3);
\draw[snake=brace, mirror snake] (9.3,-4.2) -- (10.7,-4.2) node[midway,below] {$\delta_0$};
\draw[snake=brace, mirror snake] (0.3,-4.2) -- (1.7,-4.2) node[midway,below] {$\delta_3$};
\draw[snake=brace, mirror snake] (3.3,-4.2) -- (4.7,-4.2) node[midway,below] {$\delta_2$};
\draw[snake=brace, mirror snake] (6.3,-4.2) -- (7.7,-4.2) node[midway,below] {$\delta_1$};
\draw[->, dashed] (1.5,-4.4) to node [sloped,below] {$\tau_{3}$} (3.5,-4.4);
\draw[->, dashed] (4.5,-4.4) to node [sloped,below] {$\tau_{2}$} (6.5,-4.4);
\draw[->, dashed] (7.5,-4.4) to node [sloped,below] {$\tau_{1}$} (9.5,-4.4);
\draw (1,0) node  {$d_3$};
\draw (4,0) node  {$d'_3$};
\draw (4,-1.5) node  {$d_2$};
\draw (7,-1.5) node  {$d'_2$};
\draw (7,-3) node  {$d_1$};
\draw (10,-3) node  {$d'_1$};
\end{tikzpicture}
\caption{Schematic representation of elements of $X$.}
\label{fig:associativity 3}
\end{center}
\end{figure}

We define the maps $\phi_1:X\longrightarrow X_1$ and $\phi_2:X\longrightarrow X_2$ as follows: 
\begin{eqnarray*}
\phi_1(\epsilon_1,\epsilon_2,\epsilon_3)=
\big((\epsilon_{1_{|_{d_1\cup d'_2}}},\epsilon_{2_{|_{\tau_2^{-1}(d_1\cup d'_2)}}}),(\epsilon_1\circ \epsilon_2,\epsilon_3)\big),
\end{eqnarray*}
\begin{eqnarray*}
\phi_2(\epsilon_1,\epsilon_2,\epsilon_3)=
\big((\epsilon_{2_{|_{d_2\cup d'_3}}},\epsilon_{3_{|_{\tau_3^{-1}(d_2\cup d'_3)}}}),(\epsilon_2\circ \epsilon_3,\epsilon_1)\big).
\end{eqnarray*}
Informally, the map $\phi_1$ (resp. $\phi_2$) forgets the dashed ellipse shape at the top (resp. bottom) of the third (resp. second) column of Figure \ref{fig:associativity 3}. We denote by $2a=|d_3|=|d'_3|, 2b=|d_2|=|d'_2|, 2c=|d_1|=|d'_1|, 2d=|d^{'}_3\cap d_2|$ and $2e=|d^{'}_2\cap d_1|$. We prove the following lemma:
\begin{lem}\label{surjection de phi}
The map $\phi_1$ is well defined and surjective. For any element $\big((\widetilde{\alpha_1},\widetilde{\alpha_2}),(\widetilde{\widetilde{\alpha_1}\circ \widetilde{\alpha_2}},\widetilde{\alpha_3})\big)\in X_1$, we have:
$$|\phi_1^{-1}\Big(\big((\widetilde{\alpha_1},\widetilde{\alpha_2}),(\widetilde{\widetilde{\alpha_1}\circ \widetilde{\alpha_2}},\widetilde{\alpha_3})\big)\Big)|=
2^{a-d-f}\cdot(n-b-c+e)_{(a-d-f)},$$
where $2f=|(\widetilde{d_2}\setminus d_2)\cap d^{'}_3|$.\\
Note that $f$, unlike $a,b,c,d,e$, depends on the element of $X_1$ that we consider.
\end{lem}
\begin{proof}
For any $(\epsilon_1,\epsilon_2,\epsilon_3)\in X$, we must check that $\phi_1\big((\epsilon_1,\epsilon_2,\epsilon_3)\big)$ is in $X_1$. 
Set $$(\widetilde{\alpha_1},\widetilde{\alpha_2})=(\epsilon_{1_{|_{d_1\cup d'_2}}},\epsilon_{2_{|_{\tau_2^{-1}(d_1\cup d'_2)}}}).$$

First $(\widetilde{\alpha_1},\widetilde{\alpha_2})$ is in $P_{\alpha_1}(n)\times P_{\alpha_2}(n)$ since $(\epsilon_1,\epsilon_2)$ is in $P_{\alpha_1}(n)\times P_{\alpha_2}(n)$. Then, the codomain of $\widetilde{\alpha_2}$ and the domain of $\widetilde{\alpha_1}$ are $d_1\cup d'_2$. Thus, $(\widetilde{\alpha_1},\widetilde{\alpha_2})$ is in $E_{\alpha_1}^{\alpha_2}(n)$. 

Second, it is easy to check that $(\epsilon_1\circ \epsilon_2,\epsilon_3)$ is in $P_{\widetilde{\alpha_1}\circ \widetilde{\alpha_2}}(n)\times P_{\alpha_3}(n)$. Then, to be in $E_{\widetilde{\alpha_1}\circ \widetilde{\alpha_2}}^{\alpha_3}(n)$, $(\epsilon_1\circ \epsilon_2,\epsilon_3)$ must verify the condition that $\delta_{2}=d'_{3}\cup \tau_{2}^{-1}(d_1\cup d'_2)$, which is given by condition \eqref{delta2}. 

Thus, $\phi_1$ is well defined.\medskip

Now, fix $\big((\widetilde{\alpha_1},\widetilde{\alpha_2}),(\widetilde{\widetilde{\alpha_1}\circ \widetilde{\alpha_2}},\widetilde{\alpha_3})\big)\in X_1$.
We will count the number of its pre-images by $\phi_1$.
 To construct an element of $\phi_1^{-1}\Big(\big((\widetilde{\alpha_1},\widetilde{\alpha_2}),(\widetilde{\widetilde{\alpha_1}\circ \widetilde{\alpha_2}},\widetilde{\alpha_3})\big)\Big)$, we have only to build $\epsilon_1, \epsilon_2$ and $\delta_1$ since the other elements $\epsilon_3,\delta_0,\delta_2$ and $\delta_3$ are determined by $\widetilde{\alpha_1}, \widetilde{\alpha_2}, \widetilde{\alpha_3}$ and $\widetilde{\widetilde{\alpha_1}\circ \widetilde{\alpha_2}}$. First, to build $\delta_1$, we must extend  $d_1\cup d^{'}_2$ by adding pairs of form $\rho(k)$ to obtain a set which has the same cardinality as $\delta_2=\widetilde{d_2}\cup d_3^{'}$ which is $2a+2b+2c-2e-2d-2f$. We have $|d_1\cup d^{'}_2|=2b+2c-2e$, so the number of possible ways to extend $d_1\cup d^{'}_2$ is the number of choices of $2a+2b+2c-2e-2d-2f-(2b+2c-2e)=2a-2d-2f$ elements among $2n-(2b+2c-2e)$. Since our choice must respect the condition that the extended set is in $\mathbf{P}_n$, this number is:
$$\begin{pmatrix}
n-(b+c-e)\\
a-d-f
\end{pmatrix}.$$ 
Once the set $d_1\cup d^{'}_2$ is extended, we should extend $\widetilde{\sigma_1}$ to $\tau_1$ (we have the definition domain $\delta_0$ and the arrival domain $\delta_1$) by sending the pairs of form $\rho(k)$ to pairs with same form. The number of ways to do so is:
$$2^{a-d-f}\cdot(a-d-f)!.$$ 
After extending $\widetilde{\sigma_1}$ to $\tau_1$, we get immediately $\tau_2$ because $\tau_1\circ\tau_2=\widetilde{\widetilde{\sigma_1}\circ \widetilde{\sigma_2}}$ is given. Thus, the cardinality of the set $\phi_1^{-1}\Big(\big((\widetilde{\alpha_1},\widetilde{\alpha_2}),(\widetilde{\widetilde{\alpha_1}\circ \widetilde{\alpha_2}},\widetilde{\alpha_3})\big)\Big)$ is equal to:\\
$$\begin{pmatrix}
n-(b+c-e)\\
a-d-f
\end{pmatrix}\cdot2^{a-d-f}\cdot(a-d-f)!=
2^{a-d-f}\cdot(n-b-c+e)_{(a-d-f)}.
$$
\end{proof}
In the same way, we can prove the following lemma.
\begin{lem}
The map $\phi_2$ is well defined and surjective. For any element $\big((\widetilde{\alpha_2},\widetilde{\alpha_3}),(\widetilde{\alpha_1},\widetilde{\widetilde{\alpha_2}\circ \widetilde{\alpha_3}})\big)\in X_2$, we have:
$$|\phi_2^{-1}\Big(\big((\widetilde{\alpha_2},\widetilde{\alpha_3}),(\widetilde{\alpha_1},\widetilde{\widetilde{\alpha_2}\circ \widetilde{\alpha_3}})\big)\Big)|=
2^{c-e-g}\cdot(n-a-b+d)_{(c-e-g)},$$
where, $2g=|(\widetilde{d^{'}_{2}}\setminus d^{'}_{2})\cap d_{1}|$.
\end{lem}
We can also verify using the notations above that:
 $$| E_{\alpha_1}^{\alpha_2}(n)|=2^{b+c-2e}\cdot(n-c)_{(b-e)}\cdot(n-b)_{(c-e)},$$
 $$| E_{\alpha_2}^{\alpha_3}(n)|=2^{a+b-2d}\cdot(n-b)_{(a-d)}\cdot(n-a)_{(b-d)},$$
and
$$| E_{\widetilde{\alpha_1}\circ \widetilde{\alpha_2}}^{\alpha_3}(n)|=2^{a+b+c-2d-e-2f}\cdot(n-b-c+e)_{(a-d-f)}\cdot(n-a)_{(b+c-d-e-f)},$$
$$| E_{\alpha_1}^{\widetilde{\alpha_2}\circ \widetilde{\alpha_3}}(n)|=2^{a+b+c-d-2e-2g}\cdot(n-a-b+d)_{(c-e-g)}\cdot(n-c)_{(a+b-d-e-g)}.$$
The products $(\alpha_1\pr \alpha_2)\pr \alpha_3$ and $\alpha_1\pr (\alpha_2\pr \alpha_3)$ given in equations \eqref{produit 1} and \eqref{produit 2} as sums over $X_1$ and $X_2$ can be written as sums over the set $X$ as follows:
\begin{eqnarray*}
(\alpha_1\pr \alpha_2)\pr \alpha_3 =
\sum_{(\epsilon_1,\epsilon_2,\epsilon_3)\in X}&&\frac{1}{2^{2a+2b+2c-3d-3e-3f}}\cdot ((n-b-c+e)_{(a-d-f)})^2\\
&&\cdot(n-a)_{(b+c-d-e-f)}\cdot(n-c)_{(b-e)}\cdot(n-b)_{(c-e)}\\
&&\epsilon_1\circ\epsilon_2\circ\epsilon_3,
\end{eqnarray*} 
and
\begin{eqnarray}\label{second product}
\alpha_1\pr (\alpha_2\pr\alpha_3)=
\sum_{(\epsilon_1,\epsilon_2,\epsilon_3)\in X}&&\frac{1}{2^{2a+2b+2c-3d-3e-3g}}\cdot ((n-a-b+d)_{(c-e-g)})^2\\
&&\cdot (n-c)_{(a+b-d-e-g)}\cdot(n-a)_{(b-d)}\cdot (n-b)_{(a-d)} \nonumber \\
&&\epsilon_1\circ\epsilon_2\circ\epsilon_3. \nonumber
\end{eqnarray} 
For any positive integer $n$, we have the following easy identities:
\begin{eqnarray*}
(n-b-c+e)_{(a-d-f)}.(n-c)_{(b-e)}&=&(n-c)_{(a+b-d-e-f)}\\
(n-a)_{(b-d)}.(n-a-b+d)_{(c-e-f)}&=&(n-a)_{(b+c-d-e-f)}\\
(n-b)_{(c-e)}.(n-b-c+e)_{(a-d-f)}&=&(n-b)_{(a-d)}.(n-a-b+d)_{(c-e-f)}.
\end{eqnarray*}
Thus, the product $(\alpha_1\pr \alpha_2)\pr \alpha_3$ can be written as follows:
\begin{eqnarray}\label{first product}
(\alpha_1\pr \alpha_2)\pr \alpha_3 =
\sum_{(\epsilon_1,\epsilon_2,\epsilon_3)\in X}&&\frac{1}{2^{2a+2b+2c-3d-3e-3f}}\cdot ((n-a-b+d)_{(c-e-f)})^2\\
&&\cdot(n-c)_{(a+b-d-e-f)}\cdot(n-a)_{(b-d)}\cdot (n-b)_{(a-d)} \nonumber \\
&&\epsilon_1\circ\epsilon_2\circ\epsilon_3. \nonumber
\end{eqnarray}
For any element of $X$, the equality $|\delta_2|=|\delta_1|$ can be written $2c+2b-2e+2a-2d-2f=2a+2b-2d+2c-2e-2g$, so we have $f=g$. Comparing \eqref{second product} and \eqref{first product}, we see that products $(\alpha_1\pr \alpha_2)\pr \alpha_3$ and $\alpha_1\pr (\alpha_2\pr \alpha_3)$ are equal, therefore we get the associativity. 

\subsection{Action of $\mathcal{B}_n\times \mathcal{B}_n$ on $\mathcal{D}_n$}
In this section, we build the algebra $\mathcal{A}_n$ as the algebra of invariant elements by an action of $\mathcal{B}_n\times \mathcal{B}_n$ on $\mathcal{D}_n$.
\begin{definition}
The group $\mathcal{B}_n\times \mathcal{B}_n$ acts on $Q_n$ by:
$$(a,b)\bullet(\sigma,d,d')=(a\sigma b^{-1},b(d),a(d')),$$
for any $(a,b)\in \mathcal{B}_n\times \mathcal{B}_n$ and $(\sigma,d,d') \in Q_n$.
\end{definition}
\begin{obs}
Two partial bijections are in the same orbit if and only if they have the same coset-type.
\end{obs}
We can extend this action by linearity to get an action of $\mathcal{B}_n\times \mathcal{B}_n$ on $\mathcal{D}_n$.
\begin{lem}\label{compatibility of E}
For any three permutations $a,b$ and $c$ of $\mathcal{B}_n$ and for any partial bijections $\alpha_1,\alpha_2$ of $n$, the set $E_{\alpha_1}^{\alpha_2}(n)$ is in bijection with $E_{(a,b)\bullet \alpha_1}^{(b,c)\bullet \alpha_2}(n)$.
\end{lem}
\begin{proof}
We can check easily that the two following functions:
$$\begin{array}{ccccc}
\Theta & : & E_{\alpha_1}^{\alpha_2}(n) & \to & E_{(a,b)\bullet \alpha_1}^{(b,c)\bullet \alpha_2}(n) \\
& & (\widetilde{\alpha_1},\widetilde{\alpha_2}) & \mapsto & ((a,b)\bullet \widetilde{\alpha_1},(b,c)\bullet \widetilde{\alpha_2}) \\
\end{array},$$
and
$$\begin{array}{ccccc}
\Psi & : & E_{(a,b)\bullet \alpha_1}^{(b,c)\bullet \alpha_2}(n) & \to & E_{\alpha_1}^{\alpha_2}(n) \\
& & (\beta_1,\beta_2) & \mapsto & ((a^{-1},b^{-1})\bullet \beta_1,(b^{-1},c^{-1})\bullet \beta_2) \\
\end{array},$$
are well defined. Besides, they are inverse from each other:
\begin{eqnarray*}
\Psi\Big(\Theta\big((\widetilde{\alpha_1},\widetilde{\alpha_2})\big)\Big)&=&\Psi\Big(\big((a,b)\bullet \widetilde{\alpha_1},(b,c)\bullet \widetilde{\alpha_2}\big)\Big)\\
&=&\Big((a^{-1},b^{-1})\bullet (a,b)\bullet \widetilde{\alpha_1},(b^{-1},c^{-1})\bullet (b,c)\bullet \widetilde{\alpha_2}\Big)\\
&=&(\widetilde{\alpha_1},\widetilde{\alpha_2}),
\end{eqnarray*}
and, similarly,
\begin{eqnarray*}
\Theta\Big(\Psi\big(\beta_1,\beta_2\big)\Big)&=&(\beta_1,\beta_2).
\end{eqnarray*}
Thus $\Theta$ defines a bijection between $E_{\alpha_1}^{\alpha_2}(n)$ and $E_{(a,b)\bullet \alpha_1}^{(b,c)\bullet \alpha_2}(n)$ with inverse $\Psi$.
\end{proof}
It follows from this lemma that the action $\bullet$ is compatible with the product of $Q_n$. Namely, we can prove the following corollary.
\begin{cor}\label{compatibility}
For any $(a,b,c)\in \mathcal{B}_n^3$ and for any partial bijections $\alpha_1,\alpha_2$ of $n$, we have: 
\begin{equation}\label{equation 4}
(a,c)\bullet(\alpha_1\pr \alpha_2)=((a,b)\bullet \alpha_1)\pr ((b,c)\bullet \alpha_2).
\end{equation}
\end{cor}
\begin{proof} If $(\widetilde{\alpha_1},\widetilde{\alpha_2}) \in E_{\alpha_1}^{\alpha_2}(n)$, we have:

\begin{eqnarray*}
\big((a,b)\bullet \widetilde{ \alpha_1})\circ ((b,c)\bullet\widetilde{\alpha_2}\big)&=&\big(a\widetilde{\sigma_1}b^{-1},b(\widetilde{d_1}),a(\widetilde{d'_1})\big)\circ \big(b\widetilde{\sigma_2}c^{-1},c(\widetilde{d_2}),b(\widetilde{d'_2})\big)\\
&=&(a\widetilde{\sigma_1}\widetilde{\sigma_2}c^{-1},c(\widetilde{d_2}),a(\widetilde{d'_1}))\\
&=&(a,c)\bullet (\widetilde{\alpha_1}\circ \widetilde{\alpha_2}).
\end{eqnarray*}
Then, we can write:
\begin{eqnarray*}
(a,c)\bullet(\alpha_1\pr \alpha_2)&=&\frac{1}{| E_{\alpha_1}^{\alpha_2}(n)|}\sum_{(\widetilde{\alpha_1},\widetilde{\alpha_2}) \in E_{\alpha_1}^{\alpha_2}(n)} (a,c)\bullet (\widetilde{\alpha_1}\circ \widetilde{\alpha_2})\\
&=&\frac{1}{| E_{\alpha_1}^{\alpha_2}(n)|}\sum_{(\widetilde{\alpha_1},\widetilde{\alpha_2}) \in E_{\alpha_1}^{\alpha_2}(n)}  \big((a,b)\bullet \widetilde{ \alpha_1}\big)\circ \big((b,c)\bullet\widetilde{\alpha_2}\big)\\
&=&\frac{1}{| E_{(a,b)\bullet \alpha_1}^{(b,c)\bullet\alpha_2}(n)|}\sum_{(\widetilde{(a,b)\bullet \alpha_1},\widetilde{(b,c)\bullet \alpha_2}) \in E_{(a,b)\bullet \alpha_1}^{(b,c)\bullet \alpha_2}(n)}  \widetilde{(a,b)\bullet  \alpha_1}\circ \widetilde{(b,c)\bullet\alpha_2}\\
&=&((a,b)\bullet \alpha_1)\pr ((b,c)\bullet \alpha_2)).
\end{eqnarray*}
\end{proof}
We consider the set $\mathcal{A}_n$ of invariant elements by the action of $\mathcal{B}_n\times \mathcal{B}_n$ on $\mathcal{D}_n$:
$$\mathcal{A}_n=\mathcal{D}_n^{\mathcal{B}_n\times \mathcal{B}_n}=\lbrace x\in \mathcal{D}_n ~|~ (a,b)\bullet x=x~~\text{for any } (a,b)\in \mathcal{B}_n\times \mathcal{B}_n\rbrace.$$
For every partition $\lambda$ such that $|\lambda|\leq n$, we define the set $A_{\lambda, n}$ to be the set of all partial bijections $\alpha$ of $n$ such that $ct(\alpha)=\lambda$. The sum of all elements in $A_{\lambda, n}$ is denoted by $S_{\lambda, n}$.
\begin{prop}\label{basis of An}
The set $\mathcal{A}_n$ is an algebra with basis the elements $(S_{\lambda, n})_{|\lambda|\leq n}$.
\end{prop}
\begin{proof}
For every $ (a,b)\in \mathcal{B}_n\times \mathcal{B}_n$, and for every $x,y \in \mathcal{A}_n$, we have by linearity:
$$(a,b)\bullet(x\pr y)= ((a,id)\bullet x)\pr ((id,b)\bullet y)=x\pr y.$$
So $\mathcal{A}_n$ is an algebra.\\
Any element $x\in \mathcal{D}_n$ writes $\displaystyle{x= \sum_{k=1}^{n}\sum_{d,d^{'}\in \mathbf{P}_n \atop{|d|=|d^{'}|=2k}}\sum_{\sigma:d\rightarrow d^{'}\atop{\text{bijection}}}c_{(\sigma,d,d^{'})}(\sigma,d,d^{'})}$. If, furthermore $x$ is in $\mathcal{A}_n$, then for every $(a,b)\in \mathcal{B}_n\times \mathcal{B}_n$ we have: $$\sum_{k=1}^{n}\sum_{d,d^{'}\in \mathbf{P}_n\atop{|d|=|d^{'}|=2k}}\sum_{\sigma:d\rightarrow d^{'}\atop{\text{bijection}}}c_{(\sigma,d,d^{'})}(a\sigma b^{-1},b(d),a(d^{'}))=\sum_{k=1}^{n}\sum_{d,d^{'}\in \mathbf{P}_n\atop{|d|=|d^{'}|=2k}}\sum_{\sigma:d\rightarrow d^{'}\atop{\text{bijection}}}c_{(\sigma,d,d^{'})}(\sigma,d,d^{'}).$$
Thus, for any $(a,b)\in \mathcal{B}_n\times \mathcal{B}_n$, we have $c_{(a\sigma b^{-1},b(d),a(d^{'}))}=c_{(\sigma,d,d^{'})}$. This means that if $x\in \mathcal{A}_n$, all partial permutations in the same orbit -- that is with the same coset-type -- have the same coefficients. Therefore, the elements $(S_{\lambda, n})_{|\lambda|\leq n}$ form a basis of $\mathcal{A}_n$.
\end{proof}
\begin{cor}\label{corollary 3.1}
If $\lambda$ and $\delta$ are two partitions such that $|\lambda|,|\delta|\leq n$, there exist unique constants $c_{\lambda\delta}^{\rho}(n)\in \mathbb{C}$ such that:
$$S_{\lambda,n}\pr S_{\delta,n}=\sum_{\rho \text{ partition} \atop {\max{(|\lambda|,|\delta|)}\leq|\rho|\leq \min{(|\lambda|+|\delta|,n)}}}c_{\lambda\delta}^{\rho}(n)S_{\rho,n}.$$
\end{cor}
\begin{proof} We only have to prove the inequalities on the size of $\rho$. Let $\alpha_1$ and $\alpha_2$ be two partial bijections of $n$ with coset-type $\lambda$ and $\delta$. By definition (see Figure \ref{fig:composition}), every partial bijection of $n$ that appears in the sum of the product $\alpha_1\pr \alpha_2$ has some coset-type $\rho$ with $|\rho|=\frac{|d_1\cup d_2'|}{2}$. But
$$ \max\big( \frac{|d_1|}{2},\frac{|d_2'|}{2}\big)=\max (|\lambda|,|\delta|) \leq |\rho|=\frac{|d_1\cup d_2'|}{2}\leq \frac{|d_1|+|d_2'|}{2}=|\lambda|+|\delta|. \qedhere$$
\end{proof}
\begin{lem}\label{image du base}
Let $\lambda$ be a partition such that $|\lambda|=r\leq n$, we have:
$$\psi_n(S_{\lambda,n})=\frac{1}{2^{n-|\lambda|}(n-|\lambda|)!}\begin{pmatrix}
n-|\bar{\lambda}|\\
m_1(\lambda)
\end{pmatrix}K_{\bar{\lambda}}(n).$$
\end{lem}
\begin{proof}We first prove the following equation:
\begin{equation}\label{projection sur S2n}
\sum_{\alpha\in A_{\lambda,n}}\sum_{\hat{\alpha}\in \mathcal{S}_{2n}\cap P_\alpha(n) }\hat{\sigma}=\begin{pmatrix}
n-|\bar{\lambda}|\\
m_1(\lambda)
\end{pmatrix}K_{\bar{\lambda}}(n).
\end{equation} 
Fix a permutation $\omega\in K_{\bar{\lambda}}(n)$, that is $\omega\in \mathcal{S}_{2n}$ and $ct(\omega)=\bar{\lambda}\cup 1^{n-|\bar{\lambda|}}$. We are looking for the number of partial bijections $\alpha\in A_{\lambda,n}$ such that $\omega$ is one of its trivial extensions. There is a unique set $S$ such that $ct(\omega_{|_{S}})=\bar{\lambda}$. We call this set the support of $\omega$ and denote it $\supp(\omega)$. The following condition is necessary so that $\omega$ is a trivial extension of $\alpha$: $\supp(\omega)\subseteq d$ and $\alpha_{|_{\supp(\omega)}}$ must be equal to $\omega_{|_{\supp(\omega)}}$. Thus the partial bijections $\alpha$ we are looking for are the restrictions of $\omega$ to sets $\supp(\omega)\sqcup x$, with $|\supp(\omega)\sqcup x|=2|\lambda|$. Since $|\supp(\omega)|=2|\bar{\lambda}|$, one has the necessarily $|x|=2(|\lambda|-|\bar{\lambda}|)$. So the number of such $\alpha$ is $\begin{pmatrix}
n-|\bar{\lambda}|\\
|\lambda|-|\bar{\lambda}|
\end{pmatrix}=\begin{pmatrix}
n-|\bar{\lambda}|\\
m_1(\lambda)
\end{pmatrix}$. This ends the proof of equation \eqref{projection sur S2n}.\\
By applying $\psi_n$ to $S_{\lambda,n}$, we get:
\begin{eqnarray*}
\psi_n(S_{\lambda,n})&=&\psi_n\big(\sum_{\alpha\in A_{\lambda,n}}\alpha\big)\\
&=&\frac{1}{2^{n-|\lambda|}(n-|\lambda|)!}\sum_{\alpha\in A_{\lambda,n}}\sum_{\hat{\alpha}\in \mathcal{S}_{2n}\cap P_\alpha(n) }\hat{\sigma}\\
&=&\frac{1}{2^{n-|\lambda|}(n-|\lambda|)!}\begin{pmatrix}
n-|\bar{\lambda}|\\
m_1(\lambda)
\end{pmatrix}K_{\bar{\lambda}}(n).
\end{eqnarray*}
\end{proof}
This lemma implies that $\psi_n(\mathcal{A}_n)\subseteq \mathbb{C}[\mathcal{B}_n/ \mathcal{S}_{2n}\setminus \mathcal{B}_n]$. The homomorphism $\mathcal{A}_n\rightarrow \mathbb{C}[\mathcal{B}_n/ \mathcal{S}_{2n}\setminus \mathcal{B}_n]$ mentioned in Section \ref{section 2.5} is the restriction $\psi_{n_{|_{\mathcal{A}_n}}}$.
\subsection{homomorphism from $\mathcal{A}_{n+1}$ to $\mathcal{A}_n$}
This paragraph is dedicated to the proof of the following proposition:
\begin{prop}
The function $\varphi_{n}$ defined as follows:
$$\begin{array}{ccccc}
\varphi_{n} & : & \mathcal{A}_{n+1} & \to & \mathcal{A}_n \\
& & S_{\lambda,n+1} & \mapsto  & \left\{
\begin{array}{ll}
  \frac{n+1}{(n+1-|\lambda|)}S_{\lambda,n} & \qquad \mathrm{if}\quad |\lambda| < n+1 ,\\
  0 & \qquad \mathrm{if}\quad |\lambda|=n+1 ,\\
 \end{array}
 \right. \\
\end{array}$$
is a homomorphism of algebras.
\end{prop}
Let $S_{\lambda,n+1}$ and $S_{\delta,n+1}$ where $|\lambda|\leq n+1$ and $|\delta|\leq n+1$ be two basis elements of $\mathcal{A}_{n+1}$. 

If $\lambda$ (resp. $\delta$) is a partition of $n+1$, then $\varphi_{n}(S_{\lambda,n+1})$ (resp. $\varphi_{n}(S_{\delta,n+1})$) is equal to zero, and by Corollary \ref{corollary 3.1} we have:
$$S_{\lambda,n+1}\pr S_{\delta,n+1}=\sum_{\rho \text{ partition} \atop {|\rho|=n+1 }}c_{\lambda\delta}^{\rho}(n+1)S_{\rho,n+1}.$$
Note that the size of all partitions $\rho$ in the sum index of this equation is $n+1$. By applying $\varphi_n$, we get: $$\varphi_{n}(S_{\lambda,n+1}\pr S_{\delta,n+1})=\sum_{\rho \text{ partition} \atop {|\rho|=n+1 }}c_{\lambda\delta}^{\rho}(n+1)\varphi_{n}(S_{\rho,n+1})=0.$$
Thus in this case we have $\varphi_{n}(S_{\lambda,n+1}\pr S_{\delta,n+1})=\varphi_{n}(S_{\lambda,n+1})\pr \varphi_{n}(S_{\delta,n+1})$. 

In the other case ($|\lambda|\leq n$ and $|\delta|\leq n$) we have by Corollary \ref{corollary 3.1}:
$$S_{\lambda,n+1}\pr S_{\delta,n+1}=\sum_{r\leq n+1\atop{\rho\vdash r}}c_{\lambda\delta}^{\rho}(n+1)
S_{\rho,n+1}.$$
This gives us the following equation after applying $\varphi_n$:
\begin{eqnarray*}
\varphi_n(S_{\lambda,n+1}\pr S_{\delta,n+1})&=&\varphi_n\left(\sum_{r\leq n+1\atop{\rho\vdash r}}c_{\lambda\delta}^{\rho}(n+1)
S_{\rho,n+1}\right)\\
&=&\sum_{r\leq n\atop{\rho\vdash r}}c_{\lambda\delta}^{\rho}(n+1)
\frac{n+1}{(n+1-|\rho|)}S_{\rho,n}.
\end{eqnarray*}
In the other hand, we have:
\begin{eqnarray*}
\varphi_n(S_{\lambda,n+1})\pr \varphi_n(S_{\delta,n+1})&=&\frac{n+1}{(n+1-|\lambda|)}S_{\lambda,n}\pr \frac{n+1}{(n+1-|\delta|)}S_{\delta,n}\\
&=&\frac{n+1}{(n+1-|\lambda|)}\frac{n+1}{(n+1-|\delta|)}\sum_{r\leq n\atop{\rho\vdash r}}c_{\lambda\delta}^{\rho}(n)S_{\rho,n}.
\end{eqnarray*}
Thus, $\varphi_n$ is a homomorphism if we have the following equality for any partition $\rho$ with size at most $n$:
$$\frac{c_{\lambda\delta}^{\rho}(n+1)}{c_{\lambda\delta}^{\rho}(n)}=\frac{\frac{n+1}{(n+1-|\lambda|)}\frac{n+1}{(n+1-|\delta|)}}{\frac{n+1}{(n+1-|\rho|)}}.$$
Let $\rho$ be a partition with size at most $n$ and $\alpha$ an element of $A_{\rho,n}$. We define $H_{\lambda\delta}^{\rho}(n)$ to be the following set:
$$\lbrace \big(\alpha_1,\alpha_2\big)\in A_{\lambda,n}\times A_{\delta,n}
\text{ such that there exists } (\widetilde{\alpha_1},\widetilde{\alpha_2})\in E_{\alpha_1}^{\alpha_2}(n)\text{ with }\alpha=\widetilde{\alpha_1}\circ \widetilde{\alpha_2}\rbrace.
$$
This set depends on $\alpha$ by definition. However, $\alpha$ does not appear in our notation. This should not be an issue, since $\alpha$ is fixed in the whole proof.

The coefficient $c_{\lambda\delta}^{\rho}(n)$ can be written as follows:
$$c_{\lambda\delta}^{\rho}(n)=\displaystyle{\sum_{(\alpha_1,\alpha_2)\in H_{\lambda\delta}^{\rho}(n)}}\frac{1}{|E_{\alpha_1}^{\alpha_2}(n)|}.$$
Similarly, we have:
$$c_{\lambda\delta}^{\rho}(n+1)=\displaystyle{\sum_{(\alpha_1,\alpha_2)\in H_{\lambda\delta}^{\rho}(n+1)}}\frac{1}{|E_{\alpha_1}^{\alpha_2}(n+1)|}.$$
By equation \eqref{cardinal de E}, if $(\alpha_1,\alpha_2)\in H_{\lambda\delta}^{\rho}(n)$, we have:
$$|E_{\alpha_1}^{\alpha_2}(n)|=2^{2|\rho|-|\lambda|-|\delta|}(n-|\lambda|)_{(|\rho|-|\lambda|)}(n-|\delta|)_{(|\rho|-|\delta|)}.$$
Similarly, if $(\alpha_1,\alpha_2)\in H_{\lambda\delta}^{\rho}(n+1)$, we have:
$$|E_{\alpha_1}^{\alpha_2}(n+1)|=2^{2|\rho|-|\lambda|-|\delta|}(n+1-|\lambda|)_{(|\rho|-|\lambda|)}(n+1-|\delta|)_{(|\rho|-|\delta|)}.$$
Thus, we get:
$$c_{\lambda\delta}^{\rho}(n)=\frac{|H_{\lambda\delta}^{\rho}(n)|}{2^{2|\rho|-|\lambda|-|\delta|}(n-|\lambda|)_{(|\rho|-|\lambda|)}(n-|\delta|)_{(|\rho|-|\delta|)}},$$
and
$$c_{\lambda\delta}^{\rho}(n+1)=\frac{|H_{\lambda\delta}^{\rho}(n+1)|}{2^{2|\rho|-|\lambda|-|\delta|}(n+1-|\lambda|)_{(|\rho|-|\lambda|)}(n+1-|\delta|)_{(|\rho|-|\delta|)}}.$$
This gives us after simplification:
$$\frac{c_{\lambda\delta}^{\rho}(n+1)}{c_{\lambda\delta}^{\rho}(n)}=\frac{|H_{\lambda\delta}^{\rho}(n+1)|}{|H_{\lambda\delta}^{\rho}(n)|}\cdot \frac{n+1-|\rho|}{n+1-|\lambda|}\cdot \frac{n+1-|\rho|}{n+1-|\delta|}$$
We will now evaluate the quotient $\frac{|H_{\lambda\delta}^{\rho}(n+1)|}{|H_{\lambda\delta}^{\rho}(n)|}$. Let $\bm{u}=(u_1, u'_1,u_2,u'_2)$ be an element of $\mathbf{P}_{n}^{4}$ such that:
\begin{eqnarray}
&&u_2\subseteq d,\label{condition 1}\\
&&u'_{1}\subseteq d',\label{condition 2}\\
&&|u_1|= |u'_1|=2|\lambda|,\label{condition 3}\\
&&|u_2|= |u'_2|=2|\delta|,\label{condition 4}\\
&&|u'_2\cup u_1|=2|\rho|.\label{condition 5}
\end{eqnarray} 
We introduce \begin{eqnarray*}
N_{\bm{u}}&=&\lbrace \big(h_1=(f_1,u_1,u'_1),h_2=(f_2,u_2,u'_2)\big)\in A_{\lambda,n}\times A_{\delta,n}\\
&&\text{such that there exists }(\widetilde{h_1},\widetilde{h_2})\in E_{h_1}^{h_2}(n) \text{ with } \alpha=\widetilde{h_1}\circ \widetilde{h_2}\rbrace.
\end{eqnarray*}
Its elements are represented on Figure \ref{fig:elements of Nu}. The set $H_{\lambda\delta}^{\rho}(n)$ is the disjoint union of all $N_{\bm{u}}$ with $\bm{u}$ satisfying the above conditions.
\begin{figure}[htbp]
\begin{center}
\begin{tikzpicture},fill opacity=0.5]
\draw[] (0,0) ellipse (1cm and 2cm)++(8,0) ellipse (1cm and 2cm);
\draw (0,0) node {$d$};
\draw (8,0) node {$d'$};
\draw[->] (1,0) to node [sloped,below] {$\sigma$} (7,0);
\draw[] (0,1) ellipse (.4cm and .8cm)++(8,0) ellipse (.4cm and .8cm);
\draw (0,1) node {$u_2$};
\draw (8,1) node {$u'_1$};
\draw[] (4,3) ellipse (.4cm and .8cm);
\draw[] (4,2) ellipse (.4cm and .8cm);
\draw (4,3) node {$u'_2$};
\draw (4,2) node {$u_1$};
\draw[->,dashed] (.4,1) to node [sloped,above] {$f_2$} (3.6,3);
\draw[->,dashed] (4.4,2) to node [sloped,above] {$f_1$} (7.6,1);
\end{tikzpicture}
\caption{Schematic representation of elements of $N_{\bm{u}}$ }
\label{fig:elements of Nu}
\end{center}
\end{figure}
\begin{lem}\label{bijection Nu and Nv}
Let $\bm{v}=(v_1, v'_1,v_2,v'_2)$ be an element of $\mathbf{P}_{n}^{4}$ satisfying conditions above. If $v'_1=u'_1$ and $v_2=u_2$, then there exists a bijection between $N_{\bm{u}}$ and $N_{\bm{v}}$.
\end{lem}
\begin{proof}
We take any permutation $b\in \mathcal{B}_n$, such that $b(u_1)=v_1$ and $b(u'_2)=v'_2$. Such a permutation exists because $|u_1|=|v_1|$, $|u'_2|=|v'_2|$ and $|u'_2\cup u_1|=|v'_2\cup v_1|$. We associate to a pair $(h_1,h_2)$  in $N_{\bm{u}}$ the pair $((id,b)\bullet h_1,(b,id)\bullet h_2)$. We check that the image lies in $N_{\bm{v}}$:
\begin{eqnarray*}
\big((id,b)\bullet h_1,(b,id)\bullet h_2\big)&=&\big((f_1b^{-1},b(u_1),u'_1),(bf_2,u_2,b(u'_2))\big)\\
&=&\big((f_1b^{-1},v_1,u^{'}_1),(bf_2,u_2,v'_2)\big)\in A_{\lambda,n}\times A_{\delta,n}.
\end{eqnarray*}
We can check easily that $\big((id,b)\bullet \widetilde{h_1},(b,id)\bullet \widetilde{h_2}\big)\in E_{(id,b)\bullet h_1}^{(b,id)\bullet h_2}(n)$, and we have:
\begin{eqnarray*}
(id,b)\bullet \widetilde{h_1}\circ (b,id)\bullet \widetilde{h_2}&=&(\widetilde{f_1}b^{-1},b(\widetilde{u_1}),\widetilde{u'_1})\circ (b\widetilde{f_2},\widetilde{u_2},b(\widetilde{u'_2}))\\
&=&(\widetilde{f_1}b^{-1}b\widetilde{f_2},\widetilde{u_2},\widetilde{u'_1})\\
&=&\widetilde{h_1}\circ \widetilde{h_2}\\
&=&\alpha .
\end{eqnarray*}
It is then easy to check that this defines a bijection between $N_{\bm{u}}$ and $N_{\bm{v}}$. Details are the same as in Lemma \ref{compatibility of E}.
\end{proof}
Therefore the cardinality of $N_{\bm{u}}$ depends only on $u'_1$ and $u_2$. We denote it by $f(u'_1,u_2)$. If we denote by $\bm{U}$ the set of vectors $\bm{u}\in \mathbf{P}_n^4$ satisfying conditions \eqref{condition 1} to \eqref{condition 5}, the set $H_{\lambda\delta}^{\rho}(n)$ can be written as follows: $$\displaystyle{H_{\lambda\delta}^{\rho}(n)=\bigsqcup_{\bm{u}\in \bm{U}} N_{\bm{u}}}.$$
Using Lemma \ref{bijection Nu and Nv}, we obtain:
$$|H_{\lambda\delta}^{\rho}(n)|=\sum_{\bm{u}\in \bm{U}}|N_{\bm{u}}|=\sum_{u'_1,u_2}\sum_{u_1,u'_2}f(u'_1,u_2).$$
The first (resp. second) summation indexes are vectors $u'_1$ and $u_2$ (resp. $u_1$ and $u'_2$) satisfying conditions \eqref{condition 1} to \eqref{condition 4} (resp. \eqref{condition 3} to \eqref{condition 5}).
Since $N_{\bm{u}}$ depends only on $u'_1$ and $u_2$, we get:
$$|H_{\lambda\delta}^{\rho}(n)|=\sum_{u'_1,u_2}f(u'_1,u_2)k_n,$$ 
where $k_n$ is the number of possible choices of vectors $u_1$ and $u'_2$ satisfying conditions \eqref{condition 3} to \eqref{condition 5}. There are $\begin{pmatrix}
n\\
|\lambda|
\end{pmatrix}$ sets $u_1\in \mathbf{P}_n$ that fulfill \eqref{condition 3}. Once $u_1$ is chosen, it remains $\begin{pmatrix}
|\lambda|\\
|\lambda|+|\delta|-|\rho|
\end{pmatrix}\cdot \begin{pmatrix}
n-|\lambda|\\
|\rho|-|\lambda|
\end{pmatrix}$ ways to choose $u'_2$ with conditions \eqref{condition 4} and \eqref{condition 5}. The first binomial is the number of possible choices of $u_1\cap u'_2$ and the second one is the number of possible choices of $u'_2\setminus u_1$. Then, we have:
$$k_n=\begin{pmatrix}
n\\
|\lambda|
\end{pmatrix}\cdot \begin{pmatrix}
n-|\lambda|\\
|\rho|-|\lambda|
\end{pmatrix}\cdot \begin{pmatrix}
|\lambda|\\
|\lambda|+|\delta|-|\rho|
\end{pmatrix}.$$
Thus, the cardinality of $H_{\lambda\delta}^{\rho}(n)$ is:
$$|H_{\lambda\delta}^{\rho}(n)|=\begin{pmatrix}
n\\
|\lambda|
\end{pmatrix}\cdot \begin{pmatrix}
n-|\lambda|\\
|\rho|-|\lambda|
\end{pmatrix}\cdot \begin{pmatrix}
|\lambda|\\
|\lambda|+|\delta|-|\rho|
\end{pmatrix}\sum_{u'_1,u_2}f(u'_1,u_2).$$
The summation index does not depend on $n$ because $u'_1$ and $u_2$ should fulfill conditions \eqref{condition 1} and \eqref{condition 2}.
Similarly, we obtain:
$$|H_{\lambda\delta}^{\rho}(n+1)|=\begin{pmatrix}
n+1\\
|\lambda|
\end{pmatrix}\cdot \begin{pmatrix}
n+1-|\lambda|\\
|\rho|-|\lambda|
\end{pmatrix}\cdot \begin{pmatrix}
|\lambda|\\
|\lambda|+|\delta|-|\rho|
\end{pmatrix}\sum_{u'_1,u_2}f(u'_1,u_2),$$
which gives us:
$$\frac{|H_{\lambda\delta}^{\rho}(n+1)|}{|H_{\lambda\delta}^{\rho}(n)|}=\frac{\begin{pmatrix}
n+1\\
|\delta|
\end{pmatrix}\begin{pmatrix}
n+1-|\delta|\\
|\rho|-|\delta|
\end{pmatrix}}{\begin{pmatrix}
n\\
|\delta|
\end{pmatrix}\begin{pmatrix}
n-|\delta|\\
|\rho|-|\delta|
\end{pmatrix}}=\frac{n+1}{n+1-|\rho|}.$$
Thus, we have:
\begin{eqnarray}\label{formula of c}
\frac{c_{\lambda\delta}^{\rho}(n+1)}{c_{\lambda\delta}^{\rho}(n)}&=&\frac{n+1}{n+1-|\rho|}\cdot \frac{n+1-|\rho|}{n+1-|\lambda|}\cdot \frac{n+1-|\rho|}{n+1-|\delta|}\\
&=&\frac{\frac{n+1}{(n+1-|\lambda|)}\frac{n+1}{(n+1-|\delta|)}}{\frac{n+1}{(n+1-|\rho|)}}. \nonumber
\end{eqnarray}
This proves that $\varphi_n$ is a homomorphism of algebras.
\subsection{Projective limits}
In this paragraph, we consider the projective limit $\mathcal{A}_\infty$ of the sequence $(\mathcal{A}_n)$. We prove in Proposition \ref{proposition 3.4} that every element of $\mathcal{A}_\infty$ is written in a unique way as infinite linear combination of elements indexed by partitions.\\

First, from equation \eqref{formula of c}, we can get the following Corollary:
\begin{cor}\label{c}
Let $\lambda$, $\delta$ and $\rho$ be three partitions such that $$\max{(|\lambda|,|\delta|)}\leq |\rho|\leq |\lambda|+|\delta|.$$ For every $n\geq |\rho|$, we have:
$$c_{\lambda\delta}^{\rho}(n)=\frac{c_{\lambda\delta}^{\rho}(|\rho|)}{\begin{pmatrix}
|\rho|\\
|\lambda|
\end{pmatrix}\begin{pmatrix}
|\rho|\\
|\delta|
\end{pmatrix}}\cdot \frac{\begin{pmatrix}
n\\
|\lambda|
\end{pmatrix}\begin{pmatrix}
n\\
|\delta|
\end{pmatrix}}{\begin{pmatrix}
n\\
|\rho|
\end{pmatrix}}.$$
\end{cor}
\begin{proof}
We proceed by induction on $n$. For $n = |\rho|$, we have the equality. Assume we have the equality for some $n\geq |\rho|$ and let us prove it for $n+1$.
By equation \eqref{formula of c}, we have:$$\frac{c_{\lambda\delta}^{\rho}(n+1)}{c_{\lambda\delta}^{\rho}(n)}=\frac{\frac{n+1}{(n+1-|\lambda|)}\frac{n+1}{(n+1-|\delta|)}}{\frac{n+1}{(n+1-|\rho|)}}.$$ 
This gives us the following equality, using the induction hypothesis :
\begin{eqnarray*}
c_{\lambda\delta}^{\rho}(n+1)&=&\frac{c_{\lambda\delta}^{\rho}(|\rho|)}{\begin{pmatrix}
|\rho|\\
|\lambda|
\end{pmatrix} \begin{pmatrix}
|\rho|\\
|\delta|
\end{pmatrix}}\cdot\frac{\begin{pmatrix}
n\\
|\lambda|
\end{pmatrix}\begin{pmatrix}
n\\
|\delta|
\end{pmatrix}}{\begin{pmatrix}
n\\
|\rho|
\end{pmatrix}}\cdot \frac{\frac{n+1}{(n+1-|\lambda|)}\frac{n+1}{(n+1-|\delta|)}}{\frac{n+1}{(n+1-|\rho|)}}\\
&=&\frac{c_{\lambda\delta}^{\rho}(|\rho|)}{\begin{pmatrix}
|\rho|\\
|\lambda|
\end{pmatrix}\begin{pmatrix}
|\rho|\\
|\delta|
\end{pmatrix}}\cdot \frac{\begin{pmatrix}
n+1\\
|\lambda|
\end{pmatrix}\begin{pmatrix}
n+1\\
|\delta|
\end{pmatrix}}{\begin{pmatrix}
n+1\\
|\rho|
\end{pmatrix}}.
\end{eqnarray*}
\end{proof}
Let $\mathcal{A}_{\infty}$ be the projective limit of $(\mathcal{A}_n,\varphi_n)$:
$$\mathcal{A}_{\infty}=\lbrace (a_n)_{n\geq 1} \mid \text{for every } n\geq 1, a_n\in \mathcal{A}_n\text{ and } \varphi_n(a_{n+1})=a_n\rbrace.$$
\begin{lem}\label{Projective limit}An element $a=(a_n)_{n\geq 1}$ is in $\mathcal{A}_{\infty}$ if and only if there exists a family $(x^{a}_{\lambda})_{\lambda~partition}$ of elements of $\mathbb{C}$ such that for every $n\geq 1$, $a_n=\displaystyle{\sum_{\lambda \text{ partition}\atop{|\lambda|\leq n}}\frac{x^{a}_{\lambda}}{\begin{pmatrix}
n\\
|\lambda|
\end{pmatrix}}{S}_{\lambda,n}}$.
\end{lem}
\begin{proof}\
Let $a=(a_n)_{n\geq 1}$ be a sequence in $\mathcal{A}_{\infty}$, $a_n\in \mathcal{A}_n$ for every $n\geq 1$. By Proposition \ref{basis of An}, the elements $(S_{\lambda,n})_{\lambda\vdash r\leq n}$ form a basis of $\mathcal{A}_n$, thus for every $n\geq 1$ and every partition $\lambda$ such as $|\lambda|\leq n$, there exists a scalar $a_\lambda(n)\in \mathbb{C}$ such that \begin{displaymath}
a_n=\sum_{\lambda \text{ partition}\atop {|\lambda|\leq n}}a_{\lambda}(n)S_{\lambda,n}.
\end{displaymath}
The condition $\varphi_n(a_{n+1})=a_n$ can be written as follows:
$$\varphi_n\left(\sum_{\lambda \text{ partition}\atop{|\lambda|\leq n+1}}a_{\lambda}(n+1)S_{\lambda,n+1}\right)=\sum_{\lambda \text{ partition}\atop{|\lambda|\leq n}}a_{\lambda}(n)S_{\lambda,n}.$$
Using the definition of $\varphi_n$, we can simplify this equality to obtain:
$$\sum_{\lambda \text{ partition}\atop{|\lambda|\leq n}}a_{\lambda}(n+1)\frac{n+1}{n+1-|\lambda|}S_{\lambda,n}=\sum_{\lambda \text{ partition}\atop{|\lambda|\leq n}}a_{\lambda}(n)S_{\lambda,n}.$$
By considering the coefficients of $S_{\lambda,n}$ we get that for every partition $\lambda$ such that $|\lambda|\leq n$, we have:
$$\frac{a_{\lambda}(n+1)}{a_{\lambda}(n)}=\frac{n+1-|\lambda|}{n+1}.$$
After an immediate induction, we get :
$$a_{\lambda}(n)=\frac{a_{\lambda}(|\lambda|)}{\begin{pmatrix}
n\\
|\lambda|
\end{pmatrix}}.$$ 
Set $x^a_\lambda=a_\lambda(|\lambda|)$, this proves the "only if" statement. Converse is obvious.
\end{proof}
For every partition $\lambda$, we define the sequence $T_{\lambda}$ as follows:
$$T_{\lambda}=(T_{\lambda})_{n\geq 1}=\left\{
\begin{array}{ll}
  0 & \qquad \mathrm{if}\quad n< |\lambda|,\\\\
  \frac{1}{\begin{pmatrix}
n\\
|\lambda|
\end{pmatrix}}{S}_{\lambda,n} & \qquad \mathrm{if}\quad  n\geq |\lambda| .\\
 \end{array}
 \right.$$
From Lemma \ref{Projective limit}, we obtain directly the following proposition:
\begin{prop}\label{proposition 3.4}
 Every element $a \in \mathcal{A}_\infty$ is written in a unique way as infinite linear combination of elements $T_{\lambda}$.
 \end{prop}
This proposition shows that the algebra $\mathcal{A}_\infty$ satisfies the second property required in  Section \ref{section 2.5}. In particular, $T_{\lambda}\pr T_{\delta}$ writes as linear combination of elements $T_\rho$. We can be more precise.
\begin{cor}\label{corollary 3.2}
Let $\lambda$ and $\delta$ be two partitions, there exist unique constants $b_{\lambda\delta}^{\rho}$ such that:
$$T_{\lambda}\pr T_{\delta}=\sum_{\rho \text{ partition } \atop {\max{(|\lambda|,|\delta|)}\leq |\rho|\leq |\lambda|+|\delta|}}b_{\lambda\delta}^{\rho}T_{\rho}.$$
Moreover $b_{\lambda\delta}^{\rho}=\frac{c_{\lambda\delta}^{\rho}(|\rho|)}{\begin{pmatrix}
|\rho|\\
|\lambda|
\end{pmatrix}\begin{pmatrix}
|\rho|\\
|\delta|
\end{pmatrix}}$. In particular, it is a non-negative rational number.
\end{cor}
\begin{proof} By Proposition \ref{proposition 3.4}, $T_{\lambda}\pr T_{\delta}$ writes as linear combination of elements $T_\rho$. 
$$T_{\lambda}\pr T_{\delta}=\sum_{\rho \text{ partition }}b_{\lambda\delta}^{\rho}T_{\rho}.$$
It remains to prove how we get the conditions about the size of partitions $\rho$ that appear in the sum index and the formula for $b_{\lambda\delta}^{\rho}$. \\If $n< \max{(|\lambda|,|\delta|)}$, we have:
$$(T_{\lambda}\pr T_{\delta})_n=0.$$
Let $n\geq \max{(|\lambda|,|\delta|)}$, we use Corollary \ref{corollary 3.1} and Corollary \ref{c} to get: 
\begin{eqnarray*}
(T_{\lambda}\pr T_{\delta})_n&=&\frac{1}{\begin{pmatrix}
n\\
|\lambda|
\end{pmatrix}}{S}_{\lambda,n}\pr \frac{1}{\begin{pmatrix}
n\\
|\delta|
\end{pmatrix}}{S}_{\delta,n}\\
&=&\sum_{\rho \text{ partition}\atop{\max {(|\lambda|,|\delta|)}\leq |\rho| \leq \min {(|\lambda|+|\delta|,n )}}}\frac{c_{\lambda\delta}^{\rho}(|\rho|)}{\begin{pmatrix}
|\rho|\\
|\lambda|
\end{pmatrix}\begin{pmatrix}
|\rho|\\
|\delta|
\end{pmatrix}\begin{pmatrix}
n\\
|\rho|
\end{pmatrix}}S_{\rho,n}\\
&=&\left(\sum_{\rho \text{ partition}\atop{\max {(|\lambda|,|\delta|)}\leq |\rho| \leq |\lambda|+|\delta|}}\frac{c_{\lambda\delta}^{\rho}(|\rho|)}{\begin{pmatrix}
|\rho|\\
|\lambda|
\end{pmatrix}\begin{pmatrix}
|\rho|\\
|\delta|
\end{pmatrix}}T_{\rho}\right)_n
\end{eqnarray*}
Comparing both expressions for $T_{\lambda}\pr T_{\delta}$, this proves our proposition.
\end{proof}
\begin{ex}\label{example T2*T2}
We compute in this example the product $T_{(2)}\pr T_{(2)}.$ Using Corollary \ref{corollary 3.2}, we can write
$\displaystyle{T_{(2)}\pr T_{(2)}=\sum_{\rho \text{ partition } \atop {2\leq |\rho|\leq 4}}b_{(2)(2)}^{\rho}T_{\rho}}$, which gives us:
\begin{eqnarray*}
T_{(2)}\pr T_{(2)}&=& ~b_{(2)(2)}^{(1^2)}T_{(1^2)}+b_{(2)(2)}^{(1^3)}T_{(1^3)}+b_{(2)(2)}^{(1^4)}T_{(1^4)}\\
&& +b_{(2)(2)}^{(2)}T_{(2)}+b_{(2)(2)}^{(2,1)}T_{(2,1)}+b_{(2)(2)}^{(2,1^2)}T_{(2,1^2)}+b_{(2)(2)}^{(2^2)}T_{(2^2)}\\
&& +b_{(2)(2)}^{(3)}T_{(3)}+b_{(2)(2)}^{(3,1)}T_{(3,1)}\\
&& +b_{(2)(2)}^{(4)}T_{(4)}.
\end{eqnarray*}
The formula for $b_{\lambda\delta}^{\rho}$ given in Corollary \ref{corollary 3.2} shows that these elements can be computed using the product of $S_{(2),|\lambda|+|\delta|}\pr S_{(2),|\lambda|+|\delta|}$ in $\mathcal{A}_{|\lambda|+|\delta|}$, which is $\mathcal{A}_4$ in our case.\\
We have implemented the algebra $\mathcal{A}_n$ in \cite{Sage} and got the following equation for the product $S_{(2),4}*S_{(2),4}$ in $\mathcal{A}_4$:
$$S_{(2),4}*S_{(2),4}=96 S_{(1^2),4}+48S_{(2),4}+36S_{(3),4}+12S_{(2^2),4}.$$
Using the formulas for $c_{\lambda\delta}^\rho(|\rho|)$ and $b_{\lambda\delta}^{\rho}$ given in Corollary \ref{c} and \ref{corollary 3.2}, we obtain:
$$T_{(2)}\pr T_{(2)}=16T_{(1^2)}+8T_{(2)}+4T_{(3)}+\frac{1}{3}T_{(2^2)}.$$
\end{ex}
\begin{cor}
The set of all finite linear combinations of ($T_\lambda$), denoted by $\widetilde{\mathcal{A}_\infty}$, forms a sub-algebra of $\mathcal{A}_\infty$. The family $(T_\lambda)_{\lambda \text{ partition}}$ is a basis of $\widetilde{\mathcal{A}_\infty}$.
\end{cor}
\begin{proof}
This comes from the fact that the partitions $\rho$ indexing the sum in the product $T_\lambda \pr T_\delta$ verify:
$$|\rho|\leq |\lambda|+|\delta|.$$
\end{proof}
The algebra $\widetilde{\mathcal{A}_\infty}$ will be of interest in Section \ref{section 5}.
\section{Proof of Theorem 2.1} \label{section 4}
In the previous section, we built all algebras and homomorphisms that we need in order to prove Theorem ~\ref{Theorem 2.1}.

Let $\lambda$ and $\delta$ be two proper partitions, by Corollary \ref{corollary 3.2}, we have:
$$T_{\lambda}\pr T_{\delta}=\sum_{\rho \text{ partition}\atop {\max{(|\lambda|,|\delta|)}\leq |\rho|\leq |\lambda|+|\delta|}}b_{\lambda\delta}^{\rho}T_{\rho}.$$
Recall that this is an equality of sequences. Taking the $n$-th term, we have:
\begin{eqnarray*}
\frac{1}{\begin{pmatrix}
n\\
|\lambda|
\end{pmatrix}}{S}_{\lambda,n}\pr \frac{1}{\begin{pmatrix}
n\\
|\delta|
\end{pmatrix}}{S}_{\delta,n}&=&\sum_{\rho~\text{partition}\atop {\max{(|\lambda|,|\delta|)}\leq|\rho|\leq \min{(|\lambda|+|\delta|,n)}}}b_{\lambda\delta}^{\rho}\frac{1}{\begin{pmatrix}
n\\
|\rho|
\end{pmatrix}}{S}_{\rho,n} .
\end{eqnarray*}
By applying $\psi_n$ we obtain (see Lemma \ref{image du base}):
\begin{multline*}
\frac{1}{2^{n-|\lambda|}(n-|\lambda|)!}K_{\lambda}(n)\cdot \frac{1}{2^{n-|\delta|}(n-|\delta|)!}K_{\delta}(n)=\\
\sum_{\rho~\text{partition}\atop {\max{(|\lambda|,|\delta|)}\leq|\rho|\leq \min{(|\lambda|+|\delta|,n)}}}b_{\lambda\delta}^{\rho}\frac{\begin{pmatrix}
n\\
|\lambda|
\end{pmatrix}\begin{pmatrix}
n\\
|\delta|
\end{pmatrix}}{\begin{pmatrix}
n\\
|\rho|
\end{pmatrix}2^{n-|\rho|}(n-|\rho|)!}\begin{pmatrix}
n-|\bar{\rho}|\\
m_1(\rho)
\end{pmatrix}K_{\bar{\rho}}(n).
\end{multline*}
After simplification, we get:
\begin{multline*}
K_{\lambda}(n)\cdot K_{\delta}(n)=
\sum_{\rho\text{ partition} \atop \max{(|\lambda|,|\delta|)}\leq|\rho|\leq \min{(|\lambda|+|\delta|,n)}}b_{\lambda\delta}^{\rho}\frac{(|\rho|)_{|\bar{\rho}|}}{|\lambda|!|\delta|!}2^{n+|\rho|-|\lambda|-|\delta|}n!(n-|\bar{\rho}|)_{m_1(\rho)}K_{\bar{\rho}}(n).
\end{multline*}
\begin{fact} Any partition $\rho$ such that $|\rho|\leq \min{(|\lambda|+|\delta|,n)}$ can be written in a unique way as
$\rho=\tau \cup (1^j),$
where $\tau$ is a proper partition and $j\leq \min{(|\lambda|+|\delta|,n)}-|\tau|$.
\end{fact}
Using this fact, the product can be written as follows: 
$$K_{\lambda}(n)\cdot K_{\delta}(n)=\sum_{\tau\text{ proper partition}\atop {|\tau|\leq \min{(|\lambda|+|\delta|,n)}}}\alpha_{\lambda\delta}^{\tau}(n)K_{\tau}(n),$$
where 
\begin{eqnarray}\label{alpha}
\alpha_{\lambda\delta}^{\tau}(n)&=&\frac{1}{|\lambda|!|\delta|!}\sum_{j=0}^{\min{(|\lambda|+|\delta|,n)}-|\tau|}b_{\lambda\delta}^{\tau\cup (1^{j})}n!(n-|\tau|)_{j}(|\tau|+j)_{|\tau|}2^{n+|\tau|+j-|\lambda|-|\delta|}\\
&=&\frac{2^nn!}{|\lambda|!|\delta|!}\sum_{j=0}^{|\lambda|+|\delta|-|\tau|}b_{\lambda\delta}^{\tau\cup (1^{j})}(n-|\tau|)_{j}(|\tau|+j)_{|\tau|}2^{|\tau|+j-|\lambda|-|\delta|}. \nonumber
\end{eqnarray}
The change of sum index in the last equality comes from the fact that if $n< |\lambda|+|\delta|$, we have:
$$(n-|\tau|)_{j}=0 \text{ for any $j$ with } n-|\tau|< j\leq |\lambda|+|\delta|-|\tau|.$$
This ends the proof of Theorem \ref{Theorem 2.1}.

The polynomial of some structure coefficients is constant, especially we have the following corollary.
\begin{cor}
If $\lambda$, $\delta$ and $\rho$ are three proper partitions such that $|\rho|=|\lambda|+|\delta|$, then:
$$\alpha_{\lambda\delta}^{\rho}(n)=b_{\lambda\delta}^{\rho}\frac{|\rho|!}{|\lambda|!|\delta|!}2^nn!=c_{\lambda\delta}^\rho(|\rho|)\frac{|\lambda|!|\delta|!}{(|\lambda|+|\delta|)!}2^nn!.$$
\end{cor}
\begin{ex}
We recall that the product $T_{(2)}\pr T_{(2)}$ has been computed in Example \ref{example T2*T2}. We deduce from it the complete formula for the product $K_{(2)}(n)\cdot K_{(2)}(n)$ for every $n\geq 4$ . Using formula \eqref{alpha}, we have:
$$K_{(2)}(n)\cdot K_{(2)}(n)=2^nn!n(n-1)K_{\emptyset}(n)+2^{n}n!K_{(2)}(n)+2^{n}n!3K_{(3)}(n)+2^{n}n!2K_{(2^2)}(n).$$
\end{ex}
\section{A link with shifted symmetric functions}\label{section 5}

In \cite[Section 9]{IvanovKerov1999}, Ivanov and Kerov have given an isomorphism between the algebra of $1$-shifted symmetric functions and the algebra $\mathcal{\textbf{A}}_\infty$, which is the universal algebra that projects on the center of the symmetric group algebra $Z(\mathbb{C}[\mathcal{S}_n])$, for each $n$. In this section, using the zonal spherical functions of the Gelfand pair $(\mathcal{S}_{2n},\mathcal{B}_n)$, we prove that there is an isomorphism between the algebra of $2$-shifted symmetric functions and the algebra $\widetilde{\mathcal{A}_\infty}$.

We start with the definition of the algebra of shifted symmetric functions with coefficients in $\mathbb{C}(\alpha)$, denoted by $\Lambda^{*}(\alpha)$. An $\alpha$-\textit{shifted symmetric function} $f$ in infinitely many variables $(x_1,x_2,\cdots)$ is a family $(f_i)_{i\geq 1}$ with the two following properties:
\begin{enumerate}
\item $f_i$ is a symmetric polynomial in $(x_1-\frac{1}{\alpha},x_2-\frac{2}{\alpha},\cdots,x_i-\frac{i}{\alpha}).$
\item $f_{i+1}(x_1,x_2,\cdots,x_i,0)=f_i(x_1,x_2,\cdots,x_i).$
\end{enumerate}
The set of all shifted symmetric functions is an algebra denoted $\Lambda^{*}(\alpha)$. In \cite{Lassalle}, Lassalle gives an isomorphism between the algebra of symmetric functions with coefficients in $\mathbb{C}[\alpha]$, denoted by $\Lambda(\alpha)$, and $\Lambda^{*}(\alpha)$. We will denote this isomorphism by $sh_{\alpha}$ instead of $(\#)$, as used by Lassalle. We prefer this notation as it makes the dependence in the parameter $\alpha$ explicit.

Let $f$ be an element of $\Lambda^{*}(\alpha)$. For any partition $\lambda=(\lambda_1,\lambda_2,\cdots,\lambda_l)$, we denote by $f(\lambda)$ the value $f_{l}(\lambda_1,\lambda_2,\cdots,\lambda_{l})$. The shifted symmetric function $f$ is determined by its values on partitions, see \cite[Section 2]{1996q.alg.....8020O}.

\subsection{Gelfand pairs and zonal spherical functions.} Let $G$ be a finite group and $K$ a subgroup of $G$. We denote by $C(G,K)$ the set of all complex-valued functions on $G$ that are constant on each $K$-double coset in $G$. Namely, $$C(G,K)=\lbrace f:G\longrightarrow \mathbb{C} \text{ such that }f(kxk')=f(x) \text{ for all $x$ in $G$ and $k,k'$ in $K$} \rbrace.$$
The set $C(G,K)$ is an algebra with product defined as follows (usually called \textit{convolution product}):
$$(fg)(x)=\sum_{y\in G}f(y)g(y^{-1}x)\text{    for all $f,g$ in $C(G,K)$}.$$
The pair $(G,K)$ is said to be a \textit{Gelfand pair} if the algebra $C(G,K)$ is commutative. More details about Gelfand pairs are given in \cite[Chapter VII, 1]{McDo}. In particular, when $(G,K)$ is a Gelfand pair, the algebra $C(G,K)$ admit a relevant canonical basis $(\omega_i)$. The $\omega_i$ are called \textit{zonal spherical functions}.

\begin{prop}\label{homomorphism}
Every zonal spherical function $\omega$ of a Gelfand pair $(G,K)$ defines a homomorphism of $\mathbb{C}[K\setminus G/ K]$ to $\mathbb{C^*},$ where $\mathbb{C}[K\setminus G/ K]$ is the sub-algebra of $\mathbb{C}[G]$ of elements invariant under the $K$-double action.
\end{prop}
\begin{proof}
A zonal spherical function $\omega$ has the following property given in \cite[page 392]{McDo}:
\begin{equation}\label{zonalfunctionsproperty}
\omega(x)\omega(y)=\frac{1}{|K|}\sum_{k\in K}\omega(xky),
\end{equation}
for all $x,y\in G$. This property can be extended by linearity to the group algebra $\mathbb{C}[G]$. If $x$ and $y$ are two elements of $\mathbb{C}[K\setminus G/ K]$, then we have:
\begin{equation*}
\omega(x)\omega(y)=\frac{1}{|K|}\sum_{k\in K}\omega(xky)=\frac{1}{|K|}\sum_{k\in K}\omega(xy)=\omega(xy),
\end{equation*}
which ends the proof of the proposition.
\end{proof}

The pair $(\mathcal{S}_{2n},\mathcal{B}_n)$ is a Gelfand pair (see \cite[Chapter VII, 2]{McDo}) and its zonal spherical functions are indexed by partitions of $n$. They are denoted by $\omega^{\rho}$ and defined by: $$\omega^{\rho}(x)=\frac{1}{|\mathcal{B}_n|}\sum_{k\in \mathcal{B}_n}\chi^{2\rho}(xk),$$
for $x\in \mathcal{S}_{2n}$, where $\chi^{2\rho}$ is the character of the irreducible $\mathcal{S}_{2n}$-module corresponding to $2\rho:=(2\rho_1,2\rho_2,\cdots)$. Two permutations $x$ and $y$ in the same $\mathcal{B}_n$-double coset $K_\lambda(n)$ have the same image by $\omega^\rho$ denoted by $\omega^\rho_\lambda$.

\subsection{Jack polynomials.} The family of Jack polynomials $J_{\rho}(\alpha)$, indexed by partitions, forms a basis of $\Lambda(\alpha)$. In the basis of power sums $p_\lambda$, $J_{\rho}(\alpha)$ may be developed as follows:
\begin{equation}\label{jackpolynomial1}J_{\rho}(\alpha)=\sum_{|\lambda|=|\rho|}\theta^{\rho}_{\lambda}(\alpha)p_\lambda.
\end{equation}

Let $\lambda$ and $\rho$ be partitions with $|\rho|\geq |\lambda|$. Then the shifted symmetric function $sh_\alpha(p_\lambda)$ is related to $\theta^{\rho}_{\lambda\cup (1^{n-|\lambda|})}(\alpha)$ by the following equation given in \cite[Proposition 2]{Lassalle}:
\begin{equation}\label{lassalleequation}\alpha^{|\lambda|-l(\lambda)}sh_\alpha(p_\lambda)(\rho)=\begin{pmatrix}
|\rho|-|\lambda|+m_1(\lambda)\\
m_1(\lambda)
\end{pmatrix}z_\lambda \theta^{\rho}_{\lambda\cup (1^{n-|\lambda|})}(\alpha),
\end{equation}
where $\begin{displaystyle}z_\lambda=\prod_{i\geq 1}i^{m_i(\lambda)}m_i(\lambda)!\end{displaystyle}$. Directly from the definition of $sh_\alpha(f)$ for any symmetric function $f$, given in \cite[Eq. (3.1)]{Lassalle}, one has $sh_\alpha(p_\lambda)(\rho)=0$ if $|\rho|<|\lambda|$.

\subsection{Isomorphism between $\widetilde{\mathcal{A}_\infty}$ and $\Lambda^*(2)$.} When $\alpha=2$, Jack polynomials are related to zonal spherical functions of $(\mathcal{S}_{2n},\mathcal{B}_n)$ by the following equation (cf. \cite[page 408]{McDo}):
\begin{equation}\label{jackpolynomial2}J_{\rho}(2)=|\mathcal{B}_{n}|\sum_{|\lambda|=n}z_{2\lambda}^{-1}\omega_{\lambda}^{\rho}p_\lambda,
\end{equation}
for every partition $\rho$ of $n$. This formula can be viewed as an analogue for $\alpha=2$ of the following formula known as Frobenius formula, see \cite[page 114]{McDo}:
\begin{equation*}
s_\rho=\sum_{|\lambda|=|\rho|}z_{\lambda}^{-1}\chi_{\lambda}^{\rho}p_\lambda,
\end{equation*}
where $s_\rho$ is the Schur function. Equations (\ref{jackpolynomial1}) and (\ref{jackpolynomial2}) give us the following equality when $\alpha=2$:
\begin{equation*}\label{}
\theta^{\rho}_{\lambda}(2)=|\mathcal{B}_{|\rho|}|z_{2\lambda}^{-1}\omega_{\lambda}^{\rho}.
\end{equation*}

\begin{theorem}\label{F}
The linear mapping $F:\widetilde{\mathcal{A}_{\infty}}\longrightarrow \Lambda^*(2)$ defined by:
\begin{equation}\label{isomorphism}F(T_{\lambda})=2^{2|\lambda|-l(\lambda)}|\lambda|!\frac{sh_2(p_{\lambda})}{z_{\lambda}},
\end{equation}
is an isomorphism of algebras.
\end{theorem} 
\begin{proof}
Let $\lambda$ be a partition and $T_\lambda$ the corresponding element in $\widetilde{\mathcal{A}_\infty}$. Let $n$ be an integer, $n\geq |\lambda|$. By definition, $T_\lambda$ is a sequence and its $n$-th term $(T_\lambda)_n=\frac{1}{\begin{pmatrix}
n\\
|\lambda|
\end{pmatrix}}S_{\lambda,n}$ lies in $\mathcal{A}_n$. We project onto $\mathbb{C}[\mathcal{B}_n\setminus\mathcal{S}_{2n}/ \mathcal{B}_n]$ by appying $\psi_n$. By Lemma \ref{image du base}, we get:
\begin{eqnarray*}
\psi_n((T_\lambda)_n)&=&\frac{1}{\begin{pmatrix}
n\\
|\lambda|
\end{pmatrix}}\frac{1}{2^{n-|\lambda|}(n-|\lambda|)!}\begin{pmatrix}
n-|\bar{\lambda}|\\
m_1(\lambda)
\end{pmatrix}K_{\bar{\lambda}}(n)\\
&=&\frac{2^{|\lambda|}|\lambda|!}{2^nn!}\begin{pmatrix}
n-|\bar{\lambda}|\\
m_1(\lambda)
\end{pmatrix}K_{\bar{\lambda}}(n).
\end{eqnarray*}
For any partition $\rho$ with size equal to $n$, by applying $\omega^\rho$ to $\psi_n((T_\lambda)_n)$, we obtain:
\begin{eqnarray*}
\omega^\rho(\psi_n((T_\lambda)_n))&=&\frac{2^{|\lambda|}|\lambda|!}{2^nn!}\begin{pmatrix}
n-|\bar{\lambda}|\\
m_1(\lambda)
\end{pmatrix}|K_{\bar{\lambda}}(n)|\omega^\rho_{\bar{\lambda}\cup (1^{n-|\bar{\lambda}|})}\\
&=&\frac{2^{|\lambda|}|\lambda|!}{2^nn!}\begin{pmatrix}
n-|\bar{\lambda}|\\
m_1(\lambda)
\end{pmatrix}|K_{\bar{\lambda}}(n)|\theta^\rho_{\bar{\lambda}\cup (1^{n-|\bar{\lambda}|})}(2)\frac{1}{|\mathcal{B}_{n}|z_{2(\bar{\lambda}\cup (1^{n-|\bar{\lambda}|})}^{-1}}.
\end{eqnarray*}
We denote by $|K_{\bar{\lambda}}(n)|$ the number of permutations of $2n$ with coset-type $\bar{\lambda}\cup (1^{n-|\bar{\lambda}|})$. This number is equal to $\frac{|\mathcal{B}_n|^2}{z_{2(\bar{\lambda}\cup (1^{n-|\bar{\lambda}|})}}$ (see \cite[page 402]{McDo}). Thus, after simplification, we get:
\begin{equation*}
\omega^\rho(\psi_n((T_\lambda)_n))=2^{|\lambda|}|\lambda|!\begin{pmatrix}
n-|\bar{\lambda}|\\
m_1(\lambda)
\end{pmatrix}\theta^\rho_{\bar{\lambda}\cup (1^{n-|\bar{\lambda}|})}(2),
\end{equation*}
which together with (\ref{lassalleequation}) gives us the following equation: 
\begin{equation*}
\omega^\rho(\psi_n((T_\lambda)_n))=2^{2|\lambda|-l(\lambda)}|\lambda|!\frac{sh_2(p_\lambda)(\rho)}{z_{\lambda}}.
\end{equation*}
This formula is valid for any partition $\rho$ such that $|\rho|\geq |\lambda|$. We can check that it is also valid for any partition $\rho$ with size less than $|\lambda|$, since in this case $(T_\lambda)_{|\rho|}$ and $sh_2(p_\lambda)(\rho)$ are both equal to zero.\\
Finally, for any partition $\rho$, its image by $F(T_\lambda)$ as defined in the statement of Theorem \ref{F}, can also be written as follows:
\begin{equation*}
F(T_\lambda)(\rho)=\omega^\rho(\psi_{|\rho|}((T_\lambda)_{|\rho|})).
\end{equation*}
Let $\delta$ be a partition, for any partition $\rho$, we have:
\begin{align}
F(T_\lambda\pr T_\delta)(\rho)&=\omega^\rho(\psi_{|\rho|}((T_\lambda\pr T_\delta)_{|\rho|}))\\
&=\omega^\rho(\psi_{|\rho|}((T_\lambda)_{|\rho|}\pr (T_\delta)_{|\rho|}))&\\
&=\omega^\rho(\psi_{|\rho|}(T_\lambda)_{|\rho|}\pr \psi_{|\rho|}(T_\delta)_{|\rho|})&\text{(Proposition \ref{proposition 3.2})}\\
&=\omega^\rho(\psi_{|\rho|}(T_\lambda)_{|\rho|})\cdot \omega^\rho(\psi_{|\rho|}(T_\delta)_{|\rho|}).&
\end{align}
The last equality comes from the fact that $\omega^\rho$ defines a homomorphism of $\mathbb{C}[\mathcal{B}_n\setminus\mathcal{S}_{2n}/ \mathcal{B}_n]$ to $\mathbb{C^*}$ (Proposition \ref{homomorphism}). Hence,
$$F(T_\lambda\pr T_\delta)(\rho)=(F(T_\lambda)\cdot F(T_\delta))(\rho),$$
for any two partitions $\lambda$ and $\delta$. That means that $F$ is a homomorphism of algebras from $\widetilde{\mathcal{A}_\infty}$ to $\Lambda^*(2)$. Since $(T_\lambda)_{\lambda\text{ partition}}$ and $(sh_2(p_\lambda))_{\lambda\text{ partition}}$ are respectively bases of $\widetilde{\mathcal{A}_\infty}$ and $\Lambda^*(2)$, $F$ is actually an isomorphism of algebras. 
\end{proof}
\begin{remark}
The reader should remark while reading the proof that $F(T_{\lambda})$ could be defined by $F(T_\lambda)(\rho)=\omega^\rho(\psi_{|\rho|}((T_\lambda)_{|\rho|}))$, which would show directly that $F(T_{\lambda})$ is a homomorphism since it is the composition of homomorphisms. However, we prefer to use the definition given by the equation (\ref{isomorphism}) because it gives us explicitly the action of $F$ on the elements of basis of $\widetilde{\mathcal{A}_\infty}$.
\end{remark}

\subsection{Structure constants.} As said in the beginning of this paragraph, $sh_\alpha$ is an isomorphism between $\Lambda(\alpha)$ and $\Lambda^{*}(\alpha)$. Thus, the family $(sh_\alpha(p_\lambda))_{\lambda\text{ partition}}$ forms a linear basis of $\Lambda^{*}(\alpha)$. This allows us to write the following equation:
\begin{equation*}
sh_\alpha(p_\lambda)sh_\alpha(p_\delta)=\sum_\rho g_{\lambda,\delta}^{\rho}(\alpha)sh_\alpha(p_\rho).
\end{equation*}
It is proven in \cite{dolega2012kerov} that the coefficients $g_{\lambda,\delta}^{\rho}(\alpha)$ are polynomial in $\frac{\alpha-1}{\sqrt{\alpha}}$. This structure constants are also related to the Matching-Jack conjecture of Goulden and ackson, see\cite[Section 4.5]{dolega2012kerov}.

We are interested in the case $\alpha=2$. We proved in \ref{corollary 3.2} that the coefficients $b_{\lambda\delta}^{\rho}$ that appear in the product $T_{\lambda}\pr T_{\delta}$ are non-negative rational numbers. By applying the isomorphism $F$ given in \ref{F} to the product $T_{\lambda}\pr T_{\delta}$, we get directly the following proposition.
\begin{prop}
The coefficients $g_{\lambda,\delta}^{\rho}(2)$ are non-negative rational numbers.
\end{prop}
\section{Filtrations of the algebra $\mathcal{A}_\infty$}\label{section6}

We gave in Theorem \ref{Theorem 2.1} a polynomiality property of the structure coefficients of the Hecke algebra of the pair $(\mathcal{S}_{2n},\mathcal{B}_n)$. In order to bound the degree of these polynomials, we study in this section some filtrations of the algebra $\mathcal{A}_\infty$.

From the formula of the product of basis elements in $\mathcal{A}_\infty$, given in Corollary \ref{corollary 3.2}, we can see that the function
\begin{equation*}
\deg_1(T_\lambda)=|\lambda|
\end{equation*} 
defines a filtration on $\mathcal{A}_\infty$.

In order to obtain other filtrations on $\mathcal{A}_\infty$, we give a decomposition of any partial bijection into partial bijections with coset-type equal to $(1)$ or $(2)$. We call \textit{cycle} of length $r+1$ a partial bijection with coset-type equal to $(r+1)$, where $r$ is a positive integer. We write a cycle $\mathcal{C}$ of length $r+1$ as follows (see \cite[page 2480]{Aker20122465}):
\begin{equation}\label{genericform}
\mathcal{C}=(c_1,c_2:c_3,c_4:\cdots:c_{4r+1},c_{4r+2}:c_{4r+3},c_{4r+4}).
\end{equation}
This means that:
\begin{align*}
&\lbrace c_{2i+1},c_{2i+2}\rbrace=\rho(k_i)\text{ for some $k_i$} &\text{$i=0,\cdots,2r+1$}\\
&\mathcal{C}(c_1)=c_{4r+4}&\\
&\mathcal{C}(c_{4l+1})=c_{4l}&\text{$l=1,\cdots,r$}\\
&\mathcal{C}(c_{4l+2})=c_{4l+3}&\text{$l=0,\cdots,r$}
\end{align*}
\begin{ex}
With this notation, the longest cycle in Figure \ref{fig:coset-type}, which we draw here again for convenience,
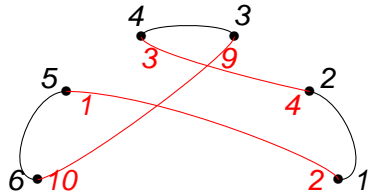
\begin{figure}[htbp]
\begin{center}
\begin{tikzpicture}
\foreach \angle / \label in
{ 0/2, 36/4, 72/9, 108/3, 144/1, 180/10}
{
\node at (\angle:2cm) {\footnotesize $\bullet$};
\draw[red] (\angle:1.7cm) node{\textsf{\label}};
}
\foreach \angle / \label in
{ 0/1, 36/2, 72/3, 108/4, 144/5, 180/6}
{\draw (\angle:2.3cm) node{\textsf{\label}};}
\draw (0:2cm) .. controls + (5mm,0) and +(.5,-.1) .. (36:2cm);
\draw[red] (36:2cm) .. controls + (-.5,.1) and + (0,-.2) .. (108:2);
\draw (108:2cm) .. controls + (0,.2) and + (0,.2) .. (72:2);
\draw[red]	(72:2)		  .. controls + (0,-.2) and + (+.5,.1) .. (180:2);
\draw	(180:2)		  .. controls + (-5mm,0) and + (-0.5,0) .. (144:2);
\draw[red]	(144:2)		  .. controls + (1.2,0) and + (0,.2) .. (0:2);
\end{tikzpicture}
\end{center}
\end{figure}\\
may be written as:
\begin{equation*}
(1,2:4,3:4,3:9,10:6,5:1,2)
\end{equation*}
\end{ex}
\begin{notation}For a partial bijection $\alpha$ and $x\in \mathcal{D}_n$, we write $\alpha \in x$ to say that the coefficient of $\alpha$ in $x$ is non-zero.
\end{notation}
\begin{obs}\label{observation}
Let $\alpha$, $\beta$ and $\gamma$ be three partial bijections such that $\gamma\in \alpha \pr \beta$. If $\alpha\in x$ and $\beta\in y$ where $x$ and $y$ are two elements of $\mathcal{D}_n$ with non-negative coefficients, then $\gamma\in x\pr y$.
\end{obs}
\begin{lem}\label{decompb}
For any cycle $\mathcal{C}$ of length $r+1$, there exist $r$ partial bijections $\tau_1,\cdots,\tau_r$ with coset-type $(2)$, such that $\mathcal{C}\in \tau_1\pr \cdots\pr \tau_r$.
\end{lem}
\begin{proof}
Let $\mathcal{C}$ be a cycle of length $r+1$ written in generic form as in (\ref{genericform}). We can check that $\mathcal{C}$ can be written as follows:
\begin{equation*}
\mathcal{C}=\mathcal{K}\circ \mathcal{J},
\end{equation*} 
where
\begin{equation*}
\mathcal{K}=(c_3,c_4:c_4,c_3)\cdots(c_{4r-5},c_{4r-4}:c_{4r-4},c_{4r-5})(c_{4r-1},c_{4r}:c_{4r+4},c_{4r+3}:c_{4r+3},c_{4r+4}:c_{4r},c_{4r-1})
\end{equation*}
and 
\begin{equation*}
\mathcal{J}=(c_1,c_2:c_3,c_4:\cdots:c_{4r-3},c_{4r-2}:c_{4r-1},c_{4r})(c_{4r+1},c_{4r+2}:c_{4r+3},c_{4r+4}).
\end{equation*}
Then, if we denote by $\tau_1$ the cycle of length $2$ of $\mathcal{K}$ and by $\mathcal{C}_1$ the cycle of length $r$ of $\mathcal{J}$, we have:
\begin{align*}
\mathcal{C}\in \tau_1\pr \mathcal{C}_1&\text{, $ct(\tau_1)=(2)$ and $ct(\mathcal{C}_1)=(r)$}.
\end{align*}
In the same way we can write:
\begin{align*}
\mathcal{C}_1\in \tau_2\pr \mathcal{C}_2&\text{ with $ct(\tau_2)=(2)$ and $ct(\mathcal{C}_2)=(r-1)$}.
\end{align*} 
Using the observation above, we get:
\begin{align*}
\mathcal{C}\in \tau_1\pr\tau_2\pr \mathcal{C}_2.
\end{align*} 
Thus, by iteration we obtain:
\begin{align*}
\mathcal{C}\in \tau_1\pr\tau_2\pr \cdots\pr\tau_r&\text{ with $ct(\tau_i)=(2)$ for all $i=1,\cdots, r$}.
\end{align*} 
This proves the lemma.
\end{proof}
\begin{lem}\label{decompb2}
For any partial bijection $\tau$ with coset-type $\rho=(\rho_1,\cdots,\rho_l)$, there exist $r$ partial bijections $\tau_1,\cdots,\tau_r$ with coset-type $(2)$ and $m_1(\rho)$ partial bijections $\beta_1,\cdots,\beta_{m_1(\rho)}$ with coset-type $(1)$ such that:
\begin{equation*}
\tau\in \tau_1\pr \cdots \pr \tau_r \pr \beta_1\pr \cdots\pr \beta_{m_1(\rho)},
\end{equation*}
where $r=|\rho|-m_1(\rho)$.
\end{lem}
\begin{proof}
Every cycle in the graph $\Gamma(\tau)$ associated to $\tau$ can be seen as a partial bijection on its own. Let $\beta_1,\cdots,\beta_{m_1(\rho)}$ the $m_1(\rho)$ partial bijections corresponding to the cycles of length $1$ of $\tau$. The other $m(\rho)=l(\rho)-m_1(\rho)$ partial bijections corresponding to the cycles of $\tau$ with length greater than $1$ are denoted by $\alpha_1,\cdots,\alpha_{m(\rho)}$. The length of $\alpha_i$ is $\rho_i$. By Lemma \ref{decompb}, for every $\alpha_i$ we can write:
\begin{align*}
\alpha_i\in \tau^i_1\pr \cdots \pr \tau^i_{\rho_i-1}&\text{ with $ct(\tau^i_j)=(2)$ for $j=1,\cdots,\rho_i-1$.}
\end{align*}
Then, since \begin{align*}
\tau\in \alpha_1\pr \cdots \pr \alpha_{m(\rho)}\pr \beta_1\pr\cdots\pr\beta_{m_1(\rho)},
\end{align*} we can write:
\begin{align*}
\tau\in \tau^1_1\pr \cdots \pr \tau^1_{\rho_1-1}\pr \cdots \pr \tau^{m(\rho)}_1\pr \cdots \pr \tau^{m(\rho)}_{\rho_{m(\rho)}-1}\pr \beta_1\pr\cdots\pr\beta_{m_1(\rho)}.
\end{align*}
The number of $\tau$'s that appear in this decomposition is equal to $(\rho_1-1)+\cdots+(\rho_{m(\rho)}-1)=|\rho|-m_1(\rho)-l(\rho)+m_1(\rho)=|\rho|-l(\rho)$, which proves the lemma.
\end{proof}
For a partial bijection $\tau$ with coset-type $\rho$, we define the following functions:
\begin{align*}
&\deg_2(\tau)=|\rho|-l(\rho)\\
&\deg_3(\tau)=|\rho|-l(\rho)+m_1(\rho).
\end{align*}
Consider the decomposition of $\tau$ given in Proposition \ref{decompb2}. We have:
\begin{align*}
&\deg_2(\tau_i)=1,\\
&\deg_2(\beta_i)=0,\\
&\deg_3(\tau_i)=1,\\
&\deg_3(\beta_i)=1.
\end{align*}
Thus
\begin{align*} 
&\deg_2(\tau)=|\rho|-l(\rho)=\deg_2(\tau_1)+\cdots +\deg_2(\tau_r)+\deg_2(\beta_1)+\cdots +\deg_2(\beta_{m_1(\rho)}),
\end{align*}
and 
\begin{align*} 
&\deg_3(\tau)=|\rho|-l(\rho)+m_1(\rho)=\deg_3(\tau_1)+\cdots +\deg_3(\tau_r)+\deg_3(\beta_1)+\cdots +\deg_3(\beta_{m_1(\rho)}).
\end{align*}
\begin{prop}\label{filtrations}
The functions 
\begin{align*}
&\deg_2(T_\rho)=|\rho|-l(\rho),\\
&\deg_3(T_\rho)=|\rho|-l(\rho)+m_1(\rho),
\end{align*}
define two filtrations on $\mathcal{A}_\infty$.
\end{prop}
\begin{proof}
Let $\sigma$ and $\tau$ be two partial bijections with coset-type $\lambda$ and $\rho$. We want to show that:
\begin{align*}
&\deg_i(\sigma\pr \tau)\leq \deg_i(\sigma)+\deg_i(\tau),&\text{for $i=1,2$,}
\end{align*}
where $\displaystyle{\deg_i(\sigma\pr \tau)=\max_{\alpha\in \sigma\pr \tau}\deg_i(\alpha)}$.\\
Because of Proposition \ref{decompb2}, $\tau$ can be decomposed into cycles of lengths $2$ and $1$. In other words, there exist $r=|\rho|-m_1(\rho)$ partial bijections $\tau_1,\cdots,\tau_r$ with coset-type $(2)$ and $m_1(\rho)$ partial bijections $\beta_1,\cdots,\beta_{m_1(\rho)}$ with coset-type $(1)$ such that:
\begin{equation*}
\tau\in \tau_1\pr \cdots \pr \tau_r \pr \beta_1\pr \cdots\pr \beta_{m_1(\rho)}.
\end{equation*}
For any partial bijection $\alpha$ such that $\alpha\in \sigma\pr \tau$, we have $\alpha\in \sigma\pr\tau_1\pr \cdots \pr \tau_r \pr \beta_1\pr \cdots\pr \beta_{m_1(\rho)}$ by Observation \ref{observation} since $\tau\in \tau_1\pr \cdots \pr \tau_r \pr \beta_1\pr \cdots\pr \beta_{m_1(\rho)}$. Thus,
\begin{align*}
\deg_i(\sigma\pr \tau)&\leq \deg_i(\sigma\pr \tau_1\pr \cdots \pr \tau_r \pr \beta_1\pr \cdots\pr \beta_{m_1(\rho)})&\text{for $i=1,2.$}
\end{align*}
If for any cycle $\mathcal{C}$ with length $2$ or $1$ and for any partial bijection $\theta$ we have,
\begin{align}\label{filt}
&\deg_i(\theta\pr \mathcal{C})\leq \deg_i(\theta)+\deg_i(\mathcal{C}), & \text{for $i=2,3$,}
\end{align} then we can write:
\begin{align*}
\deg_i(\sigma\pr \tau)&\leq \deg_i(\sigma\pr \tau_1\pr \cdots \pr \tau_r \pr \beta_1\pr \cdots\pr \beta_{m_1(\rho)})\\
&\leq \deg_i(\sigma\pr \tau_1\pr \cdots \pr \tau_r \pr \beta_1\pr \cdots\pr \beta_{m_1(\rho)-1})+\deg_i(\beta_{m_1(\rho)})\\
&\vdots \\
&\leq \deg_i(\sigma)+\deg_i(\tau_1)+ \cdots + \deg_i(\tau_r) + \deg_i(\beta_1)+ \cdots + \deg_i(\beta_{m_1(\rho)})\\
&\leq \deg_i(\sigma)+\deg_i(\tau).
\end{align*}
Therefore it is enough to prove the formula (\ref{filt}). Let $\theta$ be a partial bijection with coset-type $\delta$. If $$\mathcal{C}=(c_1,c_2:c_3,c_4)$$ is a cycle of length $1$, we have two cases:
\begin{enumerate}
\item If $\lbrace c_3, c_4\rbrace$ is in the domain of $\theta$: In this case the partial bijections that appear in the expansion of the product $\theta\pr \mathcal{C}$ have the same coset-type as $\theta$ and we have 
\begin{align*}
&\deg_i(\theta\pr \mathcal{C})\leq \deg_i(\theta)+\deg_i(\mathcal{C}), & \text{ for $i=2,3$}.
\end{align*}
\item If not, then all the partial bijections that appear in the expansion of the product $\theta\pr \mathcal{C}$ have the coset-type $\delta\cup (1)$. Then, we can check easily that $\deg_i(\theta\pr \mathcal{C})\leq \deg_i(\theta)+\deg_i(\mathcal{C})$, for $i=2,3$.
\end{enumerate} 
Now, if $\mathcal{C}=(c_1,c_2:c_3,c_4:c_5,c_6:c_7,c_8)$ is a cycle of length $2$, the figure of $\mathcal{C}$ is represented on Figure \ref{cycle C}.
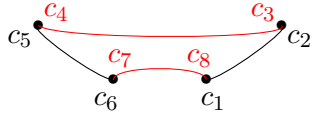
\begin{figure}[htbp]
\begin{tikzpicture}
\foreach \angle / \label in
{216/$c_4$,
252/$c_7$, 288/$c_8$, 324/$c_3$}
{
\node at (\angle:2cm) {\footnotesize $\bullet$};
\draw[red] (\angle:1.7cm) node{\textsf{\label}};
}
\foreach \angle / \label in
{216/$c_5$,
252/$c_6$, 288/$c_1$, 324/$c_2$}
{\draw (\angle:2.3cm) node{\textsf{\label}};}
\draw[red] (216:2cm) .. controls + (0.2,-0.2) and +(0.2,-0.2) .. (324:2cm);
\draw (216:2cm) .. controls + (-0.2,0) and +(-0.2,0) .. (252:2cm);
\draw[red]	(252:2cm)		    .. controls + (0,0.2) and + (0,0.2) .. (288:2);
\draw	(288:2)		    .. controls + (0.2,0) and + (0.2,0) .. (324:2);
\end{tikzpicture}
\caption{The cycle $\mathcal{C}$.}
\label{cycle C}
\end{figure}\\
We have $4$ cases. We give for each case the general result without the details of the proofs. They simply consist in computing compositions of permutations.   
\begin{enumerate}
\item $\lbrace c_3,c_4\rbrace$ and $\lbrace c_7,c_8\rbrace$ do not appear in the domain of $\theta$:

In this case, the partial bijections that appear in the expansion of the product $\theta\pr \mathcal{C}$ have coset-type $\delta\cup (2)$.
\item one of the sets $\lbrace c_3,c_4\rbrace$ and $\lbrace c_7,c_8\rbrace$ (for example $\lbrace c_3,c_4\rbrace$) appears in the domain of a cycle $\omega$ of $\theta$ and the other does not.

Suppose that $\omega$ is as represented on Figure \ref{figure in second case}.
\begin{figure}[htbp]
\begin{tikzpicture}
\foreach \angle / \label in
{ 0/$\omega_{20}$, 36/$\omega_{3}$, 72/$\omega_{4}$, 108/$\omega_{7}$, 144/$\omega_{8}$, 180/$\omega_{11}$, 216/$\omega_{12}$,
252/$\omega_{15}$, 288/$\omega_{16}$, 324/$\omega_{19}$}
{
\node at (\angle:2cm) {\footnotesize $\bullet$};
\draw[red] (\angle:1.7cm) node{\textsf{\label}};
}
\foreach \angle / \label in
{ 0/$\omega_1$, 36/$\omega_2$, 72/$\omega_5=c_3$, 108/$\omega_6=c_4$, 144/$\omega_9$, 180/$\omega_{10}$, 216/$\omega_{13}$,
252/$\omega_{14}$, 288/$\omega_{17}$, 324/$\omega_{18}$}
{\draw (\angle:2.3cm) node{\textsf{\label}};}
\draw (0:2cm) .. controls + (.2,.1) and +(.2,.1) .. (36:2cm);
\draw[red] (36:2cm) .. controls + (-.5,.1) and + (0,-.2) .. (72:2);
\draw (72:2cm) .. controls + (0,.2) and + (0,.2) .. (108:2);
\draw[red]	(108:2)		  .. controls + (0,-.2) and + (.5,.1) .. (144:2);
\draw	(144:2)		  .. controls + (-.2,0) and + (-0.2,0) .. (180:2);
\draw[red]	(180:2)		  .. controls + (.2,0) and + (0.2,0) .. (216:2);
\draw (216:2cm) .. controls + (-0.2,0) and +(-0.2,0) .. (252:2cm);
\draw[red]	(252:2cm)		    .. controls + (0,0.2) and + (0,0.2) .. (288:2);
\draw	(288:2)		    .. controls + (0.2,0) and + (0.2,0) .. (324:2);
\draw[red]	(324:2)		    .. controls + (-0.2,0) and + (-0.2,0) .. (0:2);
\end{tikzpicture}
\caption{The cycle $\omega$ of $\theta$.}
\label{figure in second case}
\end{figure}
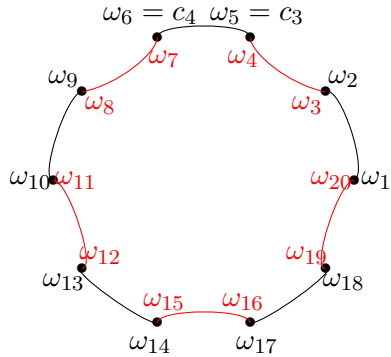\\
Then, a cycle with the form drawn on Figure \ref{form in second case} appears in the expansion of the product $\theta\pr \mathcal{C}$.
\begin{figure}[htbp]
\begin{tikzpicture}
\foreach \angle / \label in
{ 0/$\omega_{20}$, 30/$\omega_{3}$, 60/$\omega_{4}$, 90/, 120/, 150/$\omega_{7}$, 180/$\omega_{8}$,
210/$\omega_{11}$, 240/$\omega_{12}$, 270/$\omega_{15}$, 300/$\omega_{16}$, 330/$\omega_{19}$}
{
\node at (\angle:2cm) {\footnotesize $\bullet$};
\draw[red] (\angle:1.7cm) node{\textsf{\label}};
}
\foreach \angle / \label in
{ 0/, 30/, 60/$c_2$, 90/$c_1$, 120/$c_6$, 150/$c_5$, 180/,
210/, 240/, 270/, 300/, 330/}
{\draw (\angle:2.3cm) node{\textsf{\label}};}
\draw (0:2cm) .. controls + (.2,.1) and +(.2,.1) .. (30:2cm);
\draw[red] (30:2cm) .. controls + (-.5,.1) and + (0,-.2) .. (60:2);
\draw (60:2cm) .. controls + (0,.2) and + (0,.2) .. (90:2);
\draw[red]	(90:2)		  .. controls + (0,-.2) and + (0,-.2) .. (120:2);
\draw	(120:2)		  .. controls + (-.2,0) and + (-0.2,0) .. (150:2);
\draw[red]	(150:2)		  .. controls + (.2,0) and + (0.2,0) .. (180:2);
\draw (180:2cm) .. controls + (-0.2,0) and +(-0.2,0) .. (210:2cm);
\draw[red]	(210:2cm)		    .. controls + (0,0.2) and + (0,0.2) .. (240:2);
\draw	(240:2)		    .. controls + (0,-0.3) and + (0,-0.2) .. (270:2);
\draw[red]	(270:2)		    .. controls + (-0.2,0.2) and + (-0.2,0.2) .. (300:2);
\draw	(300:2)		    .. controls + (0.2,-0.2) and + (0.2,-0.2) .. (330:2);
\draw[red]	(330:2)		    .. controls + (-0.2,0) and + (-0.2,0) .. (0:2);
\end{tikzpicture}
\caption{The form of the cycle.}
\label{form in second case}
\end{figure}
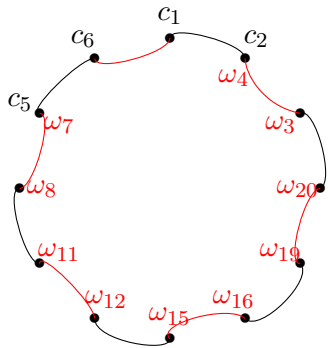\\
Note that some exterior labels are missing in this figure. To explain this, let us recall that the product $\alpha_1\pr \alpha_2$ of two partial bijections $\alpha_1$ and $\alpha_2$ is defined using an average of some partial bijections $\alpha$. When the extremity on an edge have no exterior labels, that means that the elements of any pair $\rho(k)$ different from $\lbrace c_1,c_2\rbrace$ and $\lbrace c_5,c_6\rbrace$ can be used as labels and that we shall average over all possibilities. We will also use this convention in the last two cases.\\
Thus, in this case, the coset-type of each partial bijection that appears in the expansion of the product $\theta\pr \mathcal{C}$ has the same number of parts as $\delta$ and its size is equal to $|\delta|+1$.
\item Both sets $\lbrace c_3,c_4\rbrace$ and $\lbrace c_7,c_8\rbrace$ appear in the domain of the same cycle $\omega$ of $\theta$.\\
Consider the exterior labels of this cycle $\omega$. Among them there are $c_3,c_4,c_7$ and $c_8$ and we know that $c_3$ and $c_4$ (resp. $c_7$ and $c_8$) appear consecutively. Then there are two cases that shall be considered separately. Either labels appear in cyclic order $c_3,c_4,\cdots,c_8,c_7$ or $c_3,c_4,\cdots,c_8,c_7$. These two cases are represented on Figure \ref{figure in third case} and Figure \ref{second figure in third case}.
\begin{figure}[htbp]
\begin{tikzpicture}
\foreach \angle / \label in
{ 0/$\omega_{20}$, 36/$\omega_{3}$, 72/$\omega_{4}$, 108/$\omega_{7}$, 144/$\omega_{8}$, 180/$\omega_{11}$, 216/$\omega_{12}$,
252/$\omega_{15}$, 288/$\omega_{16}$, 324/$\omega_{19}$}
{
\node at (\angle:2cm) {\footnotesize $\bullet$};
\draw[red] (\angle:1.7cm) node{\textsf{\label}};
}
\foreach \angle / \label in
{ 0/$\omega_1$, 36/$\omega_2$, 72/$\omega_5=c_3$, 108/$\omega_6=c_4$, 144/$\omega_9$, 180/$\omega_{10}$, 216/$\omega_{13}$,
252/$\omega_{14}$, 288/$\omega_{17}=c_7$, 324/$\omega_{18}=c_8$}
{\draw (\angle:2.3cm) node{\textsf{\label}};}
\draw (0:2cm) .. controls + (.2,.1) and +(.2,.1) .. (36:2cm);
\draw[red] (36:2cm) .. controls + (-.5,.1) and + (0,-.2) .. (72:2);
\draw (72:2cm) .. controls + (0,.2) and + (0,.2) .. (108:2);
\draw[red]	(108:2)		  .. controls + (0,-.2) and + (.5,.1) .. (144:2);
\draw	(144:2)		  .. controls + (-.2,0) and + (-0.2,0) .. (180:2);
\draw[red]	(180:2)		  .. controls + (.2,0) and + (0.2,0) .. (216:2);
\draw (216:2cm) .. controls + (-0.2,0) and +(-0.2,0) .. (252:2cm);
\draw[red]	(252:2cm)		    .. controls + (0,0.2) and + (0,0.2) .. (288:2);
\draw	(288:2)		    .. controls + (0.2,0) and + (0.2,0) .. (324:2);
\draw[red]	(324:2)		    .. controls + (-0.2,0) and + (-0.2,0) .. (0:2);
\end{tikzpicture}
\caption{The first possible form of $\omega$.}
\label{figure in third case}
\end{figure}
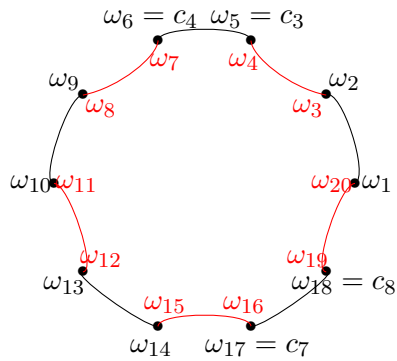
\begin{figure}[htbp]
\begin{tikzpicture}
\foreach \angle / \label in
{ 0/$\omega_{20}$, 36/$\omega_{3}$, 72/$\omega_{4}$, 108/$\omega_{7}$, 144/$\omega_{8}$, 180/$\omega_{11}$, 216/$\omega_{12}$,
252/$\omega_{15}$, 288/$\omega_{16}$, 324/$\omega_{19}$}
{
\node at (\angle:2cm) {\footnotesize $\bullet$};
\draw[red] (\angle:1.7cm) node{\textsf{\label}};
}
\foreach \angle / \label in
{ 0/$\omega_1$, 36/$\omega_2$, 72/$\omega_5=c_3$, 108/$\omega_6=c_4$, 144/$\omega_9$, 180/$\omega_{10}$, 216/$\omega_{13}$,
252/$\omega_{14}$, 288/$\omega_{17}=\color{blue}{c_8}$, 324/$\omega_{18}=\color{blue}{c_7}$}
{\draw (\angle:2.3cm) node{\textsf{\label}};}
\draw (0:2cm) .. controls + (.2,.1) and +(.2,.1) .. (36:2cm);
\draw[red] (36:2cm) .. controls + (-.5,.1) and + (0,-.2) .. (72:2);
\draw (72:2cm) .. controls + (0,.2) and + (0,.2) .. (108:2);
\draw[red]	(108:2)		  .. controls + (0,-.2) and + (.5,.1) .. (144:2);
\draw	(144:2)		  .. controls + (-.2,0) and + (-0.2,0) .. (180:2);
\draw[red]	(180:2)		  .. controls + (.2,0) and + (0.2,0) .. (216:2);
\draw (216:2cm) .. controls + (-0.2,0) and +(-0.2,0) .. (252:2cm);
\draw[red]	(252:2cm)		    .. controls + (0,0.2) and + (0,0.2) .. (288:2);
\draw	(288:2)		    .. controls + (0.2,0) and + (0.2,0) .. (324:2);
\draw[red]	(324:2)		    .. controls + (-0.2,0) and + (-0.2,0) .. (0:2);
\end{tikzpicture}
\caption{The second possible form of $\omega$.}
\label{second figure in third case}
\end{figure}
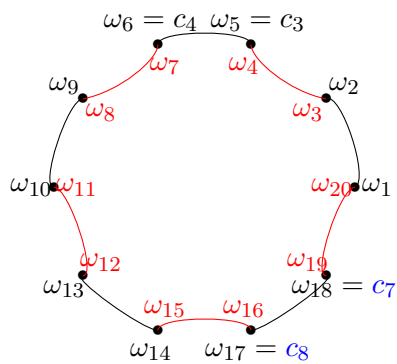\\
Note that, on Figure \ref{second figure in third case} , labels $c_7$ and $c_8$ are switched.\\
The form of cycles that appear in the expansion of the product $\theta\pr \mathcal{C}$ in each case is given on Figure \ref{form in the third case} and Figure \ref{second form in the third case}.
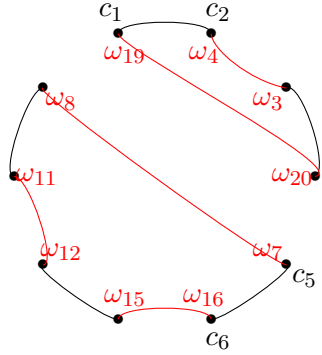
\begin{figure}[htbp]
\begin{tikzpicture}
\foreach \angle / \label in
{ 0/$\omega_{20}$, 36/$\omega_{3}$, 72/$\omega_{4}$, 108/$\omega_{19}$, 144/$\omega_{8}$, 180/$\omega_{11}$, 216/$\omega_{12}$,
252/$\omega_{15}$, 288/$\omega_{16}$, 324/$\omega_{7}$}
{
\node at (\angle:2cm) {\footnotesize $\bullet$};
\draw[red] (\angle:1.7cm) node{\textsf{\label}};
}
\foreach \angle / \label in
{ 0/, 36/, 72/$c_2$, 108/$c_1$, 144/, 180/, 216/,
252/, 288/$c_6$, 324/$c_5$}
{\draw (\angle:2.3cm) node{\textsf{\label}};}
\draw (0:2cm) .. controls + (.2,.1) and +(.2,.1) .. (36:2cm);
\draw[red] (36:2cm) .. controls + (-.5,.1) and + (0,-.2) .. (72:2);
\draw (72:2cm) .. controls + (0,.2) and + (0,.2) .. (108:2);
\draw[red]	(108:2)		  .. controls + (0,-.2) and + (.5,.1) .. (0:2);
\draw	(144:2)		  .. controls + (-.2,0) and + (-0.2,0) .. (180:2);
\draw[red]	(180:2)		  .. controls + (.2,0) and + (0.2,0) .. (216:2);
\draw (216:2cm) .. controls + (-0.2,0) and +(-0.2,0) .. (252:2cm);
\draw[red]	(252:2cm)		    .. controls + (0,0.2) and + (0,0.2) .. (288:2);
\draw	(288:2)		    .. controls + (0.2,0) and + (0.2,0) .. (324:2);
\draw[red]	(324:2)		    .. controls + (-0.2,0) and + (-0.2,0) .. (144:2);
\end{tikzpicture}
\caption{cycle correspond to first form of $\omega$.}
\label{form in the third case}
\end{figure}
\begin{figure}[htbp]
\begin{tikzpicture}
\foreach \angle / \label in
{ 0/$\omega_{20}$, 36/$\omega_{3}$, 72/$\omega_{4}$, 108/$\omega_{16}$, 144/$\omega_{15}$, 180/$\omega_{11}$, 216/$\omega_{12}$,
252/$\omega_{8}$, 288/$\omega_{7}$, 324/$\omega_{19}$}
{
\node at (\angle:2cm) {\footnotesize $\bullet$};
\draw[red] (\angle:1.7cm) node{\textsf{\label}};
}
\foreach \angle / \label in
{ 0/, 36/, 72/$c_2$, 108/$c_1$, 144/, 180/, 216/,
252/, 288/$c_5$, 324/$c_6$}
{\draw (\angle:2.3cm) node{\textsf{\label}};}
\draw (0:2cm) .. controls + (.2,.1) and +(.2,.1) .. (36:2cm);
\draw[red] (36:2cm) .. controls + (-.5,.1) and + (0,-.2) .. (72:2);
\draw (72:2cm) .. controls + (0,.2) and + (0,.2) .. (108:2);
\draw[red]	(108:2)		  .. controls + (0,-.2) and + (.5,.1) .. (144:2);
\draw	(144:2)		  .. controls + (-.2,0) and + (-0.2,0) .. (180:2);
\draw[red]	(180:2)		  .. controls + (.2,0) and + (0.2,0) .. (216:2);
\draw (216:2cm) .. controls + (-0.2,0) and +(-0.2,0) .. (252:2cm);
\draw[red]	(252:2cm)		    .. controls + (0,0.2) and + (0,0.2) .. (288:2);
\draw	(288:2)		    .. controls + (0.2,0) and + (0.2,0) .. (324:2);
\draw[red]	(324:2)		    .. controls + (-0.2,0) and + (-0.2,0) .. (0:2);
\end{tikzpicture}
\caption{cycle correspond to second form of $\omega$.}
\label{second form in the third case}
\end{figure}
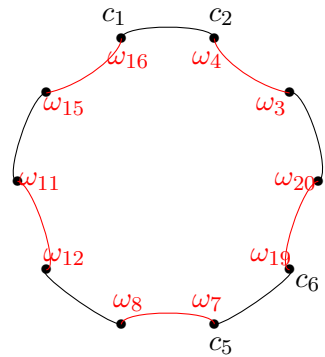\\
Thus, in the first case the cycle is cut into two cycles, then the coset-type of each partial bijection that appears in the expansion of the product $\theta\pr \mathcal{C}$ has the same size as $\delta$ and $l(\delta)+1$ parts. While in the second nothing changes and the coset-type of each partial bijection that appears in the expansion of the product $\theta\pr \mathcal{C}$ has the same size as $\delta$ and the same number of parts. 
\item the two sets $\lbrace c_3,c_4\rbrace$ and $\lbrace c_7,c_8\rbrace$ appear in the domain of two different cycles of $\theta$.\\
For example we take the two cycles represented on Figure \ref{figure of the fourth case}.\\
\begin{figure}[htbp]
\begin{center}
\begin{tikzpicture}
\foreach \angle / \label in
{ 0/$\omega_{12}$, 36/$\omega_3$, 72/$\omega_4$, 108/$\omega_7$, 144/$\omega_{8}$, 180/$\omega_{11}$, 216/$t_8$,
252/$t_3$, 288/$t_4$, 324/$t_7$}
{
\node at (\angle:2cm) {\footnotesize $\bullet$};
\draw[red] (\angle:1.7cm) node{\textsf{\label}};
}
\foreach \angle / \label in
{ 0/$\omega_1$, 36/$\omega_2$, 72/$\omega_5=c_3$, 108/$\omega_6=c_4$, 144/$\omega_{9}$, 180/$\omega_{10}$, 216/$t_1=c_7$,
252/$t_2=c_8$, 288/$t_5$, 324/$t_6$}
{\draw (\angle:2.3cm) node{\textsf{\label}};}
\draw (0:2cm) .. controls + (.2,.1) and +(.2,.1) .. (36:2cm);
\draw[red] (36:2cm) .. controls + (-.5,.1) and + (0,-.2) .. (72:2);
\draw (72:2cm) .. controls + (0,.2) and + (0,.2) .. (108:2);
\draw[red]	(108:2)		  .. controls + (0,-.2) and + (.5,.1) .. (144:2);
\draw	(144:2)		  .. controls + (-.2,0) and + (-0.2,0) .. (180:2);
\draw[red]	(180:2)		  .. controls + (.2,0.2) and + (0.2,0.2) .. (0:2);
\draw[red] (216:2cm) .. controls + (0.2,-0.2) and +(0.2,-0.2) .. (324:2cm);
\draw (216:2cm) .. controls + (-0.2,0) and +(-0.2,0) .. (252:2cm);
\draw[red]	(252:2cm)		    .. controls + (0,0.2) and + (0,0.2) .. (288:2);
\draw	(288:2)		    .. controls + (0.2,0) and + (0.2,0) .. (324:2);
\end{tikzpicture}
\caption{The two cycles.}
\label{figure of the fourth case}
\end{center}
\end{figure}
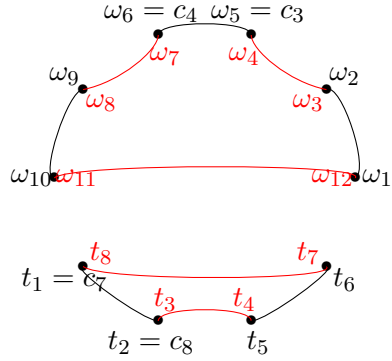
Then a cycle with the form drawn on Figure \ref{form of the fourth case} appears in the expansion of the product $\theta\pr \mathcal{C}$.\\
\begin{figure}[htbp]
\begin{center}
\begin{tikzpicture}
\foreach \angle / \label in
{ 0/$\omega_{12}$, 36/$\omega_{3}$, 72/$\omega_{4}$, 108/$t_{3}$, 144/$t_4$, 180/$t_7$, 216/$t_8$,
252/$\omega_{7}$, 288/$\omega_{8}$, 324/$\omega_{11}$}
{
\node at (\angle:2cm) {\footnotesize $\bullet$};
\draw[red] (\angle:1.7cm) node{\textsf{\label}};
}
\foreach \angle / \label in
{ 0/, 36/, 72/$c_2$, 108/$c_1$, 144/, 180/, 216/$c_6$,
252/$c_5$, 288/, 324/}
{\draw (\angle:2.3cm) node{\textsf{\label}};}
\draw (0:2cm) .. controls + (.2,.1) and +(.2,.1) .. (36:2cm);
\draw[red] (36:2cm) .. controls + (-.5,.1) and + (0,-.2) .. (72:2);
\draw (72:2cm) .. controls + (0,.2) and + (0,.2) .. (108:2);
\draw[red]	(108:2)		  .. controls + (0,-.2) and + (.5,.1) .. (144:2);
\draw	(144:2)		  .. controls + (-.2,0) and + (-0.2,0) .. (180:2);
\draw[red]	(180:2)		  .. controls + (.2,0) and + (0.2,0) .. (216:2);
\draw (216:2cm) .. controls + (-0.2,0) and +(-0.2,0) .. (252:2cm);
\draw[red]	(252:2cm)		    .. controls + (0,0.2) and + (0,0.2) .. (288:2);
\draw	(288:2)		    .. controls + (0.2,0) and + (0.2,0) .. (324:2);
\draw[red]	(324:2)		    .. controls + (-0.2,0) and + (-0.2,0) .. (0:2);
\end{tikzpicture}
\caption{The two cycles are joined.}
\label{form of the fourth case}
\end{center}
\end{figure}
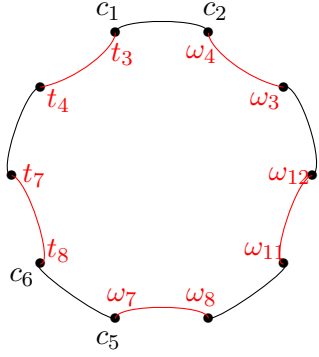
In this case the two cycles are joined to form a cycle, thus the coset-type of each partial bijection that appears in the expansion of the product $\theta\pr \mathcal{C}$ has the same size as $\delta$ and $l(\delta)-1$ parts.
\end{enumerate}
In these $4$ cases, we can check that we have $deg_i(\theta\pr \mathcal{C})\leq deg_i(\theta)+deg_i(\mathcal{C})$, for $i=2,3$ and this ends the proof of Proposition \ref{filtrations}.
\end{proof}

These filtrations allow us to get upper bounds for the degree of the polynomials $f_{\lambda\delta}^{\rho}(n)$ that appear in Theorem \ref{Theorem 2.1}.
\begin{prop}
Let $\lambda$, $\delta$ and $\rho$ be three proper partitions, the degree of $f_{\lambda\delta}^{\rho}(n)$ satisfies:
\begin{align*}
&\deg(f_{\lambda\delta}^{\rho}(n))\leq |\lambda|+|\delta|-|\rho|,\\
&\deg(f_{\lambda\delta}^{\rho}(n))\leq |\lambda|+|\delta|-|\rho|-l(\lambda)-l(\delta)+l(\rho).
\end{align*}
\end{prop}
\begin{proof}
From the proof of Theorem \ref{Theorem 2.1}, the degree of $f_{\lambda\delta}^{\rho}(n)$ is as follows: 
\begin{align}
\deg(f_{\lambda\delta}^{\rho}(n))=\max_{b_{\lambda\delta}^{\rho\cup(1^j)}\neq 0} j.
\end{align}
On the other hand, since $\deg_i$ is a filtration for $i=1,2,3$, we obtain the following inequality:
\begin{align}
\deg_i(T_\lambda\pr T_\delta)=\max_{\tau \text{ proper},j\geq 0 \atop{b_{\lambda\delta}^{\tau \cup (1^j)}\neq 0}} \deg_i(T_{\tau\cup (1^j)})\leq \deg_i(T_\lambda)+\deg_i(T_\delta).
\end{align}
Then we get the three following inequalities corresponding to the three filtrations $\deg_1$, $\deg_2$ and $\deg_3$:
\begin{align}
&\max_{\tau \text{ proper},j\geq 0 \atop{b_{\lambda\delta}^{\tau \cup (1^j)}\neq 0}} |\tau|+j\leq |\lambda|+|\delta|,&\\
&\max_{\tau \text{ proper},j\geq 0 \atop{b_{\lambda\delta}^{\tau \cup (1^j)}\neq 0}} |\tau|+j-l(\tau)-j\leq |\lambda|-l(\lambda)+|\delta|-l(\delta),&\\
&\max_{\tau \text{ proper},j\geq 0 \atop{b_{\lambda\delta}^{\tau \cup (1^j)}\neq 0}} |\tau|+j-l(\tau)-j+j\leq |\lambda|-l(\lambda)+|\delta|-l(\delta).
\end{align}
The second inequality does not give any information about the degree of $f_{\lambda\delta}^{\rho}(n)$, while using the first and the third inequality with the fact that:
\begin{align*}
\deg(f_{\lambda\delta}^{\rho}(n))=\max_{b_{\lambda\delta}^{\rho\cup(1^j)}\neq 0} j\leq \max_{\tau \text{ proper},j\geq 0 \atop{b_{\lambda\delta}^{\tau \cup (1^j)}\neq 0}}j,
\end{align*}
we get directly the result.
\end{proof}
\section*{Acknowledgement}
I would like to thank my advisors Jean-Christophe Aval and Valentin Féray for helpful discussions and suggestions and enlightening ideas of some proofs. I am also deeply grateful to them for reading and improving this article several times. 
\bibliographystyle{alpha}
\bibliography{fpsac2013paper}

\begin{thebibliography}{{Las}08}

\bibitem[AC12]{Aker20122465}
Kürşat Aker and Mahir~Bilen Can.
\newblock Generators of the hecke algebra of $({S}_{2n},{H}_n)$.
\newblock {\em Advances in Mathematics}, 231(5):2465 -- 2483, 2012.

\bibitem[Can13]{Can}
Mahir~Bilen Can.
\newblock Personal communication.
\newblock 2013.

\bibitem[Cor75]{CoriHypermaps}
R.~Cori.
\newblock Un code pour les graphes planaires et ses applications.
\newblock Number~27 in Ast{\'e}risque. Soci{\'e}t{\'e} Math{\'e}matique de
  France, 1975.
\newblock 169 pages.

\bibitem[DF12]{dolega2012kerov}
M.~Dołęga and V.~Féray.
\newblock On {K}erov polynomials for {J}ack characters.
\newblock preprint arXiv:1201.1806, 2012.

\bibitem[FH59]{FaharatHigman1959}
H.~Farahat and G.~Higman.
\newblock The centres of symmetric group rings.
\newblock {\em Proc. Roy. Soc. (A)}, 250:212--221, 1959.

\bibitem[GJ96]{GouldenJacksonLocallyOrientedMaps}
I.~P. Goulden and D.~M. Jackson.
\newblock Maps in locally orientable surfaces, the double coset algebra, and
  zonal polynomials.
\newblock {\em Can. J. Math.}, 48(3):569--584, 1996.

\bibitem[GS98]{GoupilSchaefferStructureCoef}
A.~Goupil and G.~Schaeffer.
\newblock Factoring n-cycles and counting maps of given genus.
\newblock {\em Eur. J. Comb.}, 19(7):819--834, 1998.

\bibitem[IK99]{IvanovKerov1999}
V.~Ivanov and S.~Kerov.
\newblock The algebra of conjugacy classes in symmetric groups, and partial
  permutations.
\newblock {\em Zap. Nauchn. Sem. S.-Peterburg. Otdel. Mat. Inst. Steklov.
  (POMI)}, 256(3):95--120, 1999.

\bibitem[Jam61]{James1961}
Alan~T. James.
\newblock Zonal polynomials of the real positive definite symmetric matrices.
\newblock {\em Annals of Mathematics}, 74(3):456--469, 1961.

\bibitem[JV90]{JaVi90}
D.M. Jackson and T.I. Visentin.
\newblock A character theoretic approach to embeddings of rooted maps in an
  orientable surface of given genus.
\newblock {\em Trans. AMS}, 322:343--363, 1990.

\bibitem[{Las}08]{Lassalle}
M.~{Lassalle}.
\newblock {A positivity conjecture for Jack polynomials}.
\newblock {\em Math. Res. Lett.}, 15(4):661--681, 2008.

\bibitem[Mac95]{McDo}
I.G. Macdonald.
\newblock {\em Symmetric functions and Hall polynomials}.
\newblock Oxford Univ. Press, second edition, 1995.

\bibitem[M{\'e}l13]{meliot2013partial}
Pierre-Lo{\"\i}c M{\'e}liot.
\newblock Partial isomorphisms over finite fields.
\newblock {\em arXiv preprint arXiv:1303.4313}, 2013.

\bibitem[OO97]{1996q.alg.....8020O}
Andrei Okounkov and Grigori Olshanski.
\newblock Shifted {J}ack polynomials, binomial formula, and applications.
\newblock {\em Mathematical Research Letters}, 4:69--78, 1997.

\bibitem[Sag]{Sage}
Sage mathematical software, version 4.8, http://www.sagemath.org.

\bibitem[Sol02]{Solomon2002309}
Louis Solomon.
\newblock Representations of the rook monoid.
\newblock {\em Journal of Algebra}, 256(2):309 -- 342, 2002.

\end{thebibliography}
\label{sec:biblio}
\end{document}